\documentclass{article}

\usepackage{amssymb}
\usepackage{MnSymbol}
\usepackage{proof}
\usepackage{bussproofs}
\usepackage{amsmath}
\usepackage{amsfonts}

\usepackage[show]{ed}
\usepackage{hyperref}
\usepackage{xfrac}
\usepackage{boxedminipage}

\usepackage{tikz-qtree}

\usepackage{graphicx}

\usepackage{microtype}
\usepackage{amsthm}

\newtheorem{theorem}{Theorem}[section]
\newtheorem{lemma}[theorem]{Lemma}
\newtheorem*{theorem*}{Theorem}

\theoremstyle{definition}
\newtheorem{definition}{Definition}[section]
\newtheorem{example}{Example}[section]
\theoremstyle{remark}
\newtheorem{remark}{Remark}[section]

\newcommand{\DER}{\vdash}
\newcommand{\ERGO}{\Rightarrow}

\newcommand{\seq}{\ERGO}
\newcommand{\IMPL}{\rightarrow}
\newcommand{\ET}{\wedge}

\newcommand{\VEL}{\vee}

\newcommand{\FAL}{\bot}
\newcommand{\VER}{\top}

\newcommand{\WKN}{(IW)}
\newcommand{\CTRCT}{(IC)}
\newcommand{\CUT}{(\textit{cut})}

\newcommand{\COM}{\textit{com}}

\newcommand{\LJ}{\mathrm{LJ}}
\newcommand{\HJ}{\mathrm{HLJ}}
\newcommand{\NJ}{\mathrm{NJ}}
\newcommand{\NJG}{\mathrm{NG}}
\newcommand{\GA}{\textit{GA}}
\newcommand{\NR}{\textit{Nr}}

\newcommand{\hh}{\, | \, }

\usepackage{booktabs} % For formal tables

\usepackage[ruled]{algorithm2e} % For algorithms

\SetAlFnt{\small}
\SetAlCapFnt{\small}
\SetAlCapNameFnt{\small}
\SetAlCapHSkip{0pt}
\IncMargin{-\parindent}

\begin{document}
% Title portion. Note the short title for running heads 

\title{Hypersequents and Systems of Rules: Embeddings and Applications\thanks{Supported by projects FWF START Y544-N23, WWTF MA16-028, FWF
W1255-N23, and H2020 MSCA RISE No. 689176.}}  
\date{}
\author{Agata Ciabattoni, TU Wien \\ Francesco A.\ Genco, TU Wien}

% \abstract{
% We define a bi-directional embedding between hypersequent calculi and
% a subclass of systems of rules (2-systems).  In addition to showing
% that the two proof frameworks have the same expressive power, the
% embedding allows for the recovery of the benefits of locality for
% 2-systems, analyticity results for a large class of such systems, and a
% rewriting of hypersequent rules as natural deduction rules.
% }

\maketitle

\section{Introduction}

The multitude and diversity of formalisms introduced to define analytic calculi 
has made it increasingly important to identify their
interrelationships and relative expressive power.  {\em Embeddings}
between formalisms, i.e.\ functions that take any calculus in some
formalism and yield a calculus for the same logic in another
formalism, are useful tools to prove that a formalism subsumes another
one in terms of expressiveness -- or, when bi-directional, that two
formalisms are equi-expressive. Such embeddings can also provide
useful reformulations of known calculi and allow the transfer of
certain proof-theoretic results, thus alleviating the need for
independent proofs in each system and avoiding duplicating work.
Various embeddings between formalisms have appeared in the literature, see,
 e.g.,~\cite{Revantha2016,Revantha,GorRam12AIML,Fit12,pog2010paper,pog2010book,Wan98} (and the bibliography thereof).

In this paper we introduce a bi-directional embedding between the
hypersequent formalism~\cite{Avron87} and a fragment of the system of
rules formalism~\cite{Negri:2014}.  Hypersequents are a well-studied
generalisation of sequents successfully employed to introduce analytic
proof systems for large classes of non-classical logics, see,
e.g.,~\cite{Avron:1991, Avr96, Ciabattoni:2008fk, Lahav2013,
lellmannTCS}.  Systems of rules have been recently introduced
in~\cite{Negri:2014} as a very expressive but complex formalism
capable, for example, of capturing all normal modal logics formalised
by Sahlqvist formulae.  A system of rules consists of different
(labelled) sequent rules connected by conditions on the order of their
applicability. Hence, derivations containing instances of such systems
are non-local objects, unlike hypersequent derivations.

Non-locality here has two different but closely related roles:
$(i)$ to avoid as much bureaucracy as possible in the representation
of proofs, and $(ii)$ to capture more logics.

Ad $(i)$: Natural deduction~\cite{gentzen1935} is a traditional
example of a formalism relying exclusively on formulae and non-local
effects, such as hypotheses discharge, to construct proofs. This is
particularly useful when investigating, e.g., the
computational content of proofs via a Curry--Howard
correspondence~\cite{Howard80}, but might complicate the search for and
manipulation of proofs.
Sequent~\cite{gentzen1935} and Hypersequent
calculi, by contrast, have been designed precisely with the aim to
avoid any form of non-locality.  
Locality guarantees indeed a tighter control over
proofs, thus making local proof-systems
easier to implement and to use for proving  properties of the formalised logics.
The price to pay is to deal with more complex basic objects, e.g., derivability assertions
(sequents) and their parallel composition (hypersequents). 

Ad $(ii)$: The role of non-locality to increase the expressive power of
formalisms is demonstrated in~\cite{Negri:2014}, where 
the use of systems of labelled rules allows the definition of modular analytic calculi for 
(modal and intermediate) logics whose frame conditions are beyond the  geometric fragment~\cite{sn2005}.

The system of rules formalism combines the bookkeeping
machinery of (labelled) sequent calculus with a generalised version
of the discharging mechanism of natural deduction. More precisely, a system
of rules is a set of rules that can only be applied in a
certain order and possibly share metavariables for formulae or sets of formulae.
 The word ``system'' is used in the same sense as in linear algebra, where there
are systems of equations with variables in common, and each equation
is meaningful and can be solved only if considered together with the
other equations of the system.
Consider for example the following system of sequent rules:
\[ \infer[(r)]{\Gamma \ERGO \Pi}{ \infer*{\Gamma \ERGO
\Pi}{\infer[(s)]{\Gamma_{1} \ERGO \Pi_{1}}{ \Sigma , \Gamma_{1} \ERGO
\Pi_{1}}} & \infer*{\Gamma \ERGO \Pi}{\infer[(t)]{\Sigma \ERGO }{}}}
\] Here $(s)$ and $(t)$ can only be applied above the premisses of the
rule $(r)$ and must share the metavariable  $\Sigma$. 
Hence the application of $(r)$ discharges the occurrences of
$(s)$ and $(t)$. 

The non-locality of systems of rules is twofold:
it is horizontal, because of the dependency between rules occurring in
disjoint branches; and vertical, because of rules that can only 
be applied above other rules.

A possible connection between hypersequents and systems of rules is
hinted in~\cite{Negri:2014}.  Following~\cite{CG2016} this paper
formalises and proves this intuition.  Focusing on propositional
logics intermediate
between intuitionistic and classical logic, we define a bi-directional
{\em embedding} between hypersequents and a subclass of systems of
sequent rules ({\em 2-systems}) in which the vertical non-locality is
restricted to at most two (non labelled) sequent rules.  Our
embeddings show that these two seemingly different extensions of the
sequent calculus have the same expressive power, arising from
non-local conditions for 2-systems and from bookkeeping mechanisms for
hypersequents.

From the embedding into hypersequents, 2-systems have the practical
gain of very general analyticity results. Recall indeed that
analyticity (i.e.\ the subformula property) is shown
in~\cite{Negri:2014} for systems of rules sharing only variables or atomic
formulae; while this restriction does not yield any loss of generality
in the context of labelled sequents, it does for systems of rules
operating on non-labelled sequents, e.g., defined with the aim of
directly capturing Hilbert axioms~\cite{Ciabattoni:2008fk}. Moreover
the embedding enables the introduction of new cut-free 2-systems.

The bonds unveiled by the embeddings between hypersequents and
2-systems extend further, leading to a rewriting of the former as
natural deduction systems.  As observed, e.g.\ in
\cite{Beckmann&Preining:2015}, this rewriting is a crucial step to
formalise and prove the intuition in~\cite{Avron:1991} that the
intermediate logics possessing analytic hypersequent calculi might
give rise to correponding parallel $\lambda$-calculi.  The close
relation beween systems of rules and natural deduction enables us to
define simple and modular natural deduction calculi for a large class
of intermediate logics. The calculi are obtained by extending Gentzen natural deduction calculus $\NJ$
by new rules. Similarly to sequent rules belonging to
systems, these rules can discharge other rule applications, i.e.\ they are
\emph{higher-level rules} (see e.g.,~\cite{schroederh2014}).
The results in~\cite{lics2017} for the natural deduction calculus in Example~\ref{ex:lin_nd} for G\"odel logic -- one of the best known intermediate logics  -- demonstrate the usefulness of our approach for Curry--Howard correspondences. 

\medskip

\noindent The article is structured as follows: 
Section~\ref{sec:pre} recalls the notions of hypersequent and
system of rule;
the translations between systems of rules and hypersequent rules are
presented in Section~\ref{sec:proc}; Section~\ref{sec:equivalence}
contains the embeddings between derivations:
Section~\ref{sec:sys_to_hyp} the direction from system of rules to
hypersequent derivations and Section~\ref{sec:hyp_to_sys} the inverse
direction. Sections~\ref{syst_norm_form} and~\ref{sec:prepro} introduce normal forms
for derivations containing systems of rules and hypersequents, respectively.
The final section describes the applications of the embedding,
which include the definition of new natural deduction calculi for a large
class of intermediate logics.

\medskip

\noindent The present paper extends~\cite{CG2016} in several ways: it
shows how to use the embedding to obtain natural deduction calculi,
it contains full proofs with improved techniques (e.g., the new
Section~\ref{syst_norm_form}) as
well as examples and explanations that were not included in the previous version.

\section{Preliminaries}
\label{sec:pre}
A \emph{hypersequent}~\cite{Avron87,Avron:1991} is a $\mid$-separated multiset of ordinary sequents, called
\emph{components}. 
The sequents we consider in this paper have the form $\Gamma \ERGO
\Pi$ where $\Gamma$ is a (possibly empty) multiset of formulae in the
language of intuitionistic logic and $\Pi$ contains at most one
formula.

\emph{Notation}.  Unless stated otherwise we use upper-case Greek
letters for multisets of formulae (where $\Pi$ contains at most one
element), lower-case Greek letters for formulae, and $G, H$ for
(possibly empty) hypersequents.

As with sequent calculi, the inference rules of hypersequent calculi
consist of initial hypersequents (i.e., axioms), the cut-rule as well
as logical and structural rules.  The logical and structural rules are
divided into \emph{internal} and \emph{external rules}. The internal
rules deal with formulae within one component of the conclusion.
Examples of external structural rules include external weakening $(EW)$
and external contraction $(EC)$, see Fig.~1.

Rules are usually presented as rule schemata. Concrete instances of a
rule are obtained by substituting formulae for schematic
variables. Following standard practice, we do not explicitly
distinguish between a rule and a rule schema.

Fig.~1 displays the hypersequent version  $\HJ$ of the propositional 
sequent calculus $\LJ$ for intuitionistic logic.
\begin{figure}[h]\centering
\hrule \medskip
\begin{tabular}{c}
      $\varphi \ERGO \varphi \quad  \FAL \ERGO \Pi \quad
       \; \, \vcenter{\infer[(\VEL l)]{ G \hh \Gamma , \varphi \VEL \psi \ERGO   \Pi }{ G
      \hh \Gamma ,  \varphi \ERGO \Pi \quad  G \hh \Gamma ,  \psi \ERGO
      \Pi }} \quad \vcenter{\infer[(\VEL r)]{ G \hh \Gamma \ERGO   \varphi _{1} \VEL \varphi _{2}
      }{ G \hh \Gamma  \ERGO \varphi _{i}}}$ \medskip \\ 
      \infer[( \ET l)]{ G \hh \Gamma ,  \varphi \ET \psi \ERGO \Pi }{
      G \hh \Gamma ,  \varphi, \psi \ERGO \Pi} \; \,
      \infer[(\ET r)]{ G \hh \Gamma  \ERGO \varphi \ET \psi }{ G \hh
      \Gamma \ERGO \varphi  \quad  G \hh \Gamma \ERGO \psi} \; \,     
      \infer[\WKN]{ G \hh  \varphi ,\Gamma \ERGO \Pi }{ G \hh \Gamma
      \ERGO \Pi} \medskip \\
      \infer[(\IMPL l)]{ G \hh \Gamma , \varphi \IMPL \psi
      \ERGO  \Pi }{ G \hh \Gamma \ERGO \varphi \quad  G \hh \Gamma ,
      \psi \ERGO \Pi } \; \,
      \infer[(\IMPL r)]{ G \hh \Gamma  \ERGO \varphi \IMPL \psi  }{ G \hh
      \Gamma ,  \varphi  \ERGO \psi } \; \,
      \infer[\CTRCT]{ G \hh \varphi  , \Gamma \ERGO \Pi}{ G \hh
      \varphi , \varphi ,\Gamma  \ERGO \Pi} \medskip \\
      \infer[\!\CUT]{ G \hh \Gamma , \Gamma ' \ERGO \Pi}{ G \hh
      \Gamma \ERGO \varphi \quad  G \hh  \varphi , \Gamma '  \ERGO
      \Pi } \,
      \infer[\!(EW)]{ G \hh \Gamma  \ERGO \Pi }{G} \, 
\infer[\!(EC)]{ G \hh \Gamma  \ERGO \Pi }{G \hh \Gamma  \ERGO \Pi \hh \Gamma  \ERGO \Pi}    
    \end{tabular}
\smallskip
\hrule
\label{fig:rulesHI}\caption{Rules and axioms of $\HJ$.}
\end{figure}
Note that the \emph{hyperlevel} of $\HJ$ is in fact redundant since a
hypersequent $\Gamma_1 \seq \Pi_1 \hh \dots \hh \Gamma_k \seq \Pi_k$ is derivable in $\HJ$ if and only if
$\Gamma_i \seq \Pi_i$ is derivable in $\LJ$ for some $i \in \{1, \dots ,k \}$.
Indeed, any sequent calculus can be trivially viewed as a hypersequent calculus.
The added expressive power of the latter is due to the possibility of defining new rules which act
simultaneously on several components of one or more hypersequents.
\begin{example}
\label{ex:com}
By adding to $\HJ$ the following version of the  structural rule introduced in~\cite{Avron:1991}
  \[\infer[(\COM )]{G \hh \Psi , \Gamma_{1} \ERGO \Pi _{1} \hh \Phi ,
      \Gamma _{2} \ERGO \Pi _{2} }{G \hh \Phi , \Gamma_{1} \ERGO \Pi
      _{1} & G \hh \Psi , \Gamma _{2} \ERGO \Pi _{2} }
  \]
we obtain a cut-free calculus for G\"odel logic, which is (axiomatised by) intuitionistic logic plus
the linearity axiom $(\varphi \IMPL \psi ) \VEL ( \psi \IMPL \varphi)$.

In~\cite{Avron:1991} Avron suggested that a hypersequent can be thought of as
a multiprocess. Under this interpretation, $(\COM)$ is intended to
model the exchange of information between parallel processes.

\end{example}
As the usual interpretation of the symbol ``$\hh$'' is disjunctive,
the hypersequent calculus can naturally capture 
properties (Hilbert axioms, algebraic equations\dots) that can be
expressed in a disjunctive form, see \cite{Ciabattoni:2008fk}. 

\medskip

\noindent \emph{Notation and Assumptions}.
Given a hypersequent rule $(r)$ with premisses $G \hh H_{1} \;
\dots \; G \hh H_{n}$ and conclusion $G \hh H$, we call \emph{active}
the components in the hypersequents $H_{1} , \dots ,
H_{n}, H$. We call \emph{context components} the components of $G$.
In this paper we will only consider hypersequent rules that $(i)$ are
(external) context sharing, i.e., whose premisses all contain the same
hypersequent context $G$, 
and $(ii)$ (except for $(EC)$) they have one
active component in each premiss, i.e., in which each $H_i$ is a
sequent. Note that $(i)$ is not a restriction and, in
absence of eigenvariables acting on more than one component, neither is $(ii)$;
indeed, using $(EC)$ and $(EW)$, we can always transform a rule into
an equivalent one that satisfies these conditions,
that are crucial to prove Lemma~\ref{lem:EC}.

\medskip

Systems of rules were introduced in~\cite{Negri:2014} to define
analytic labelled calculi for logics semantically characterised by
generalised geometric implications, a class of first-order formulae
that goes beyond the geometric fragment~\cite{sn2005} and includes
all frame properties that correspond to formulae in the Sahlqvist fragment.

In general, a \emph{system of rules} is a set of (possibly labelled)
sequent rules that are bound to be applied in a predetermined order
and that may share (schematic) variables or labels.  Analyticity of
systems of rules when added to a sequent or labelled sequent calculus
for classical or intuitionistic logic was proved in~\cite{Negri:2014}
for systems acting on atomic formulae or relational atoms.

The proper restriction of systems of rules that we consider in the
paper is defined below.

\begin{definition} 
\label{def:2systems}
A \emph{two-level system of rules} (\emph{2-system} for short) is a
set of sequent rules\\$\lbrace (r_{1}) , \dots , (r_{k}) , (r_{B})\rbrace
$ that can only be applied according to the following schema:
\[ \infer[(r_{B})] {\Gamma \ERGO \Pi} { \infer*{\Gamma \ERGO
\Pi}{{\mathcal D}_{1}} & \dots & \infer*{\Gamma \ERGO \Pi}{{\mathcal D}_{k}} }
\]
where each derivation ${\mathcal D}_{i}$, for $1 \leq i
\leq k$, may contain several applications of
\[
\infer[(r_i)]{ \Sigma_{0} , \Gamma ' \ERGO \Pi '}{\Sigma _{1} , \Gamma ' \ERGO \Pi ' & \dots &
  \Sigma _{n} , \Gamma ' \ERGO \Pi '}
\] 
that act on the same multisets of formulae
$\Sigma_{0}, \Sigma_{1} , \dots , \Sigma_{n}$.

The rule $(r_{B})$ is called \emph{bottom
rule}, while $(r_{1}) , \dots , (r_{k})$  \emph{top rules}.
\end{definition}
In this paper we will consider 2-systems that manipulate $\LJ$
sequents.

Given a calculus ${\mathcal C}$ and a set of rules $\mathbb{R}$, ${\mathcal C}
+ \mathbb{R}$ will denote the calculus obtained by adding the elements
of $\mathbb{R}$ to ${\mathcal C}$, and $\DER_{{\mathcal C} + \mathbb{R}}$ its
derivability relation.

\begin{example}\label{ex:sys_com} 
  The 2-system $\textit{Sys}_{(\COM ^{*})}$ in \cite{Negri:2014} for the
  linearity axiom (cf.\ Example~\ref{ex:com}) is the following ($ \varphi$ and $\psi$ are metavariables for formulae):
\[ \infer[(\COM ^{*}_{B})]{\Gamma \ERGO \Pi}{\infer*{\Gamma \ERGO
        \Pi}{\infer[(\COM ^{*} _1)]{\psi, \Gamma_{1} \ERGO \Pi
          _{1}}{ \varphi, \psi, \Gamma _{1} \ERGO \Pi _{1}}} &
      \infer*{\Gamma \ERGO \Pi}{\infer[(\COM ^{*}_2)]{\varphi , \Gamma
          _{2} \ERGO \Pi _{2}}{\varphi, \psi, \Gamma _{2} \ERGO \Pi _{2}}}}
  \] 
The analyticity of $\LJ + \textit{Sys}_{(\COM ^{*})}$ is shown in \cite{Negri:2014} for 
{\em atomic} $\varphi$ and $\psi$.
\end{example}

\begin{remark}
The above definition of 2-system differs from the one
  in~\cite{CG2016} where each rule $(r_{i})$ could only be
  applied once in $\mathcal{D}_{i}$. The following example motivates
  the adoption of the more general condition in Definition~\ref{def:2systems}.

\begin{example}
A cut-free derivation in $\LJ + \textit{Sys}_{(\COM ^{*})}$ (see Example~\ref{ex:sys_com})
of the formula  $ ((\varphi \IMPL \psi) \ET
(\varphi \IMPL \psi)) \VEL ((\psi \IMPL \varphi ) \ET (\psi \IMPL
\varphi )) $ requires two applications of
each of the top rules $(\COM ^{*} _{1})$ and $(\COM ^{*} _{2})$:
\begin{footnotesize}
  \[ \infer{\ERGO ((\varphi \IMPL \psi) \ET (\varphi \IMPL \psi)) \VEL
((\psi \IMPL \varphi ) \ET (\psi \IMPL \varphi )) } { \infer
{\ERGO((\varphi \IMPL \psi) \ET (\varphi \IMPL \psi)) \VEL ((\psi \IMPL
\varphi ) \ET (\psi \IMPL \varphi ))}{ \infer{\ERGO (\varphi \IMPL
\psi) \ET (\varphi \IMPL \psi)}{\infer{\ERGO \varphi \IMPL \psi
}{\infer[(\COM ^{*}_{1})]{ \varphi \ERGO \psi}{\infer{\varphi , \psi \ERGO
\psi}{\psi \ERGO \psi} }} & \infer{\ERGO \varphi \IMPL
\psi}{\infer[(\COM ^{*}_{1})]{ \varphi \ERGO \psi}{\infer{\varphi , \psi
\ERGO \psi}{\psi \ERGO \psi} }}}} & \infer {\ERGO ((\varphi \IMPL
\psi) \ET (\varphi \IMPL \psi)) \VEL ((\psi \IMPL \varphi ) \ET (\psi
\IMPL \varphi )) } { \infer{\ERGO (\psi \IMPL \varphi ) \ET (\psi
\IMPL \varphi )}{ \infer{\ERGO \psi \IMPL \varphi
}{\infer[(\COM ^{*}_{2})]{\psi \ERGO \varphi}{ \infer{\psi , \varphi \ERGO
\varphi}{\varphi \ERGO \varphi}}} & \infer{\ERGO \psi \IMPL
\varphi}{\infer[(\COM ^{*}_{2})]{\psi \ERGO \varphi}{ \infer{\psi , \varphi
\ERGO \varphi}{\varphi \ERGO \varphi}}} }} }
  \]
\end{footnotesize} 
\end{example}
\end{remark}

\section{From 2-systems to hypersequent rules and back}
\label{sec:proc}
We show how to rewrite a 2-system $\textit{Sys}$ into
the corresponding  hypersequent rule $\textit{Hr}_{\textit{Sys}}$;
vice versa, from a hypersequent rule $\textit{Hr}$ 
we construct the corresponding 2-system $\textit{Sys}_{\textit{Hr}}$.  The
transformation of derivations from $\HJ + \textit{Hr}$ into \mbox{$\LJ +
\textit{Sys}_{\textit{Hr}}$} (and from $\LJ + \textit{Sys}$ into $\HJ +
\textit{Hr}_{\textit{Sys}}$) is shown in
Section~\ref{sec:equivalence}.

\subsubsection*{From 2-systems to hypersequent rules}
\label{syshyp}

Given a 2-system $\textit{Sys}$ of the form
\[ \infer[(r_{B})] {\Gamma \ERGO \Pi} { \infer*{\Gamma \ERGO
\Pi} {{\mathcal D}_{1}} & \dots & \infer*{\Gamma \ERGO \Pi}{{\mathcal
D}_{k}} }
\]
 where each derivation ${\mathcal D}_{i}$, for $1
\leq i \leq k$, may contain several applications of the rule
\[ \infer[(r_{ i})] { \theta _{i}^{1}$, \dots , $ \theta _{i}^{n_{i}} ,
\Gamma _{i} \ERGO \Pi _{i}}{ \varphi _{i}^{1}, \dots , \varphi
_{i}^{l_{i}}, \Gamma _{i} \ERGO \Pi _{i} \quad \dots \quad
\psi_{i}^{1}, \dots , \psi_{i}^{m_{i}}, \Gamma _{i} \ERGO \Pi _{i}}
\] 
the corresponding hypersequent rule $\textit{Hr}_{\textit{Sys}}$ is as follows:
\[\infer[] {G \hh \theta _{1}^{1}$, \dots , $ \theta _{1}^{n_{1}} ,
\Gamma _{1} \ERGO \Pi _{1} \hh \dots \hh \theta _{k}^{1}$, \dots , $
\theta _{k}^{n_{k}} , \Gamma _{k} \ERGO \Pi _{k}} 
{M_{1} &&&& \dots  &&&& M_{k}}
\]
where $M_{i}$, for $1 \leq i \leq  k$, is the multiset of premisses 
\[ G \hh  \varphi _{i}^{1}, \dots , \varphi
_{i}^{l_{i}}, \Gamma _{i} \ERGO \Pi _{i} \quad \dots \quad G\hh 
\psi_{i}^{1}, \dots , \psi_{i}^{m_{i}}, \Gamma _{i} \ERGO \Pi _{i}\]

\begin{example}\label{ex:hyp_to_sys}
From Negri's 2-system in Example~\ref{ex:sys_com} we obtain the rule
acting on formulae $\varphi, \psi$
\[ \infer[(\COM ^{*})]{G \hh \psi, \Gamma_{1} \ERGO \Pi _{1} \hh \varphi  ,
\Gamma _{2} \ERGO \Pi _{2} }{G \hh \varphi, \psi,  \Gamma_{1} \ERGO \Pi
_{1} & G \hh \varphi, \psi , \Gamma _{2} \ERGO \Pi _{2} }
\]
 \end{example}

\subsubsection*{From hypersequent rules to 2-systems}
\label{hypsys}
Given any hypersequent rule $\textit{Hr}$ of the form
\[\infer{G \hh \Theta ^{1}_{1}$, \dots , $ \Theta _{1}^{n_{1}}
, \Gamma_{1} \ERGO \Pi _{1} \hh \dots \hh \Theta _{k}^{1}$, \dots , $ \Theta
_{k}^{n_{k}} , \Gamma_{k} \ERGO \Pi _{k}} {M_{1} &&&& \dots &&&&  M_{k}
}
\] where the sets $M_{i}$, for $1 \leq i \leq k$, constitute a
partition of the set of premisses of  $\textit{Hr}$ and each $M_{i}$
contains the premisses
\[ G \hh C ^{1}_{i} \quad \dots \quad G \hh C ^{m_{i}}_{i}\] where $
C ^{1}_{i} , \dots , C ^{m_{i}}_{i}$ are sequents. 
The corresponding 2-system $\textit{Sys}_{\textit{Hr}}$ is
\[
\infer[(r_{B})]{\Gamma \ERGO \Pi}
{\infer*{\Gamma \ERGO \Pi}{{\mathcal D}_{1}} & \dots & \infer*{\Gamma \ERGO \Pi}{{\mathcal D}_{k}}}
\] where the derivation ${\mathcal D}_{i}$, for $1
\leq i \leq k$, may contain several applications of the rule
\[
\infer[(r_{i})]{\Theta _{i}^{1}$,  \dots , $ \Theta _{i}^{n_{i}} , \Gamma _{i} \ERGO \Pi _{i}}
{ C ^{1}_{i} \quad \dots \quad C ^{m_{i}}_{i}}
\]

\begin{definition} \label{def:linked} We say that the premisses of $\textit{Hr}$  contained
in $M_{i}$, for $1 \leq i < k$, are \emph{linked} to the
component $\Theta _{i}^{1}, \dots , \Theta _{i}^{n_{1}} , \Gamma_{i}\ERGO \Pi
_{i}$ of the conclusion.
\end{definition}

\begin{example}\label{ex:sys_to_hyp}
The rewriting $\textit{Sys}_{(\COM )}$ of the rule $(\COM )$ in Example~\ref{ex:com} is 
\[ \infer[(\COM _{B})]{\Gamma \ERGO \Pi}{\infer*{\Gamma \ERGO
\Pi}{\infer[(\COM _{1})]{\Psi , \Gamma_{1} \ERGO \Pi _{1}}{\Phi  , \Gamma _{1}
\ERGO \Pi _{1}}} & \infer*{\Gamma \ERGO \Pi}{\infer[(\COM _{2})]{\Phi  ,
\Gamma _{2} \ERGO \Pi _{2}}{\Psi , \Gamma _{2} \ERGO \Pi _{2}}}}
\]
 \end{example}

\section{Embedding the two formalisms}
\label{sec:equivalence}

We introduce algorithms for transforming 2-system derivations into
hypersequent derivations and vice versa.

\subsection{From 2-systems to hypersequent derivations} 
\label{sec:sys_to_hyp}

Given any set $\mathbb{S}$ of 2-systems and set $\mathbb{H}$ of
hypersequent rules s.t.\ if $\textit{Sys} \in \mathbb{S}$ then
$\textit{Hr}_{\textit{Sys}} \in \mathbb{H}$, starting from a
derivation ${\mathcal D}$ in $\LJ +\mathbb{S}$ we construct a
derivation ${\mathcal D}'$ in $\HJ +\mathbb{H}$ of the same
end-sequent. The construction proceeds by a stepwise translation of
the rules in ${\mathcal D}$: the rules of $\LJ$ are translated into
rules of $\HJ$ -- possibly using $(EW)$ -- and, for the 2-systems in
$\mathbb{S}$, the top rules are translated into applications of the
corresponding rules in $\mathbb{H}$ -- and additional $(EW)$, if
needed -- and the bottom rules are translated into applications of
$(EC)$. To keep
track of the various translation steps, we mark the derivation ${\mathcal
D}$.  We start by marking and translating the leaves of ${\mathcal
D}$. The rules with marked premisses are then translated one by one
and the marks are moved to the conclusions of the rules. The process
is repeated until we reach and translate the root of ${\mathcal D}$.
The correct termination of the procedure is guaranteed 
when ${\mathcal D}$ satisfies the following conditions
\begin{enumerate}
\item \label{item:1} two applications of a top rule belonging to the same 2-system
instance never occur on the same path of the derivation,
\item \label{item:2} 
for each pair of 2-system instances, no top rule of one of the two
instances occurs below any top rule of the other instance (see
Definition~\ref{def:ent} as used in Lemma~\ref{lem:system_division})
\end{enumerate}
Section~\ref{syst_norm_form} shows that each 2-system derivation can be transformed into one satisfying them.

\subsubsection*{The algorithm}
\label{tr:sys-hyp} 
Input: a derivation ${\mathcal D}$ in $\LJ +\mathbb{S}$. Output:
a derivation ${\mathcal D}'$ of the same sequent in $\HJ +\mathbb{H}$.

\noindent \textbf{\emph{Translating axioms.}}  The leaves of ${\mathcal
D}$ are marked and copied as leaves of ${\mathcal D}'$.

\noindent \textbf{\emph{Translating rules.}}  Rules are translated one
by one in the following order: first the one-premiss logical and
structural rules applied to marked sequents, then the two-premiss
logical rules and bottom rules with all premisses marked, and finally
all the top rules of one 2-system instance\footnote{Condition~\ref{item:1} guarantees that all top rules of
  a 2-system instance can be translated by one hypersequent rule.}. After having translated each
rule -- or all top rules of a 2-system instance --
we remove the marks from the premisses of the translated rules and
mark their conclusions.

When we translate the top rules of a 2-system we apply
the corresponding hypersequent rule once for each possible combination
of different top rules of such system. For instance, if a 2-system
contains two applications $(r_{1})'$ and $(r_{1})''$ of one top rule,
and one application $(r_{2})$ of another top rule, we will have one
hypersequent rule application translating the pair $\langle(r_{1})',
(r_{2})\rangle$, and one hypersequent rule application translating the
pair $\langle(r_{1})'',(r_{2})\rangle$.

Since the $\LJ$ rules are particular instances of $\HJ$ rules,
we only show how to translate 2-systems. Hence, consider a 2-system $\textit{Sys} \in \mathbb{S}$  applied in ${\mathcal D}$ with
the following instances of
\begin{enumerate}
\item top rules:
\[
\infer[(r_{1})]{\Delta_{1} , \Gamma _{1} \ERGO \Pi
_{1}}{\infer*{C_{1}^{1}}{} \quad \dots \quad
\infer*{C_{1}^{m_{1}}}{}}
\qquad \dots \qquad 
\infer[(r_{k})]{\Delta _{k}, \Gamma _{k } \ERGO \Pi _{k }}{\infer*{C_{k}^{1}}{} \quad \dots \quad
\infer*{C_{k}^{m_{k}}}{}}
\]
where $ C_{1}^{1} , \dots ,
C_{1}^{m_{1}}, \dots , C_{k}^{1} ,
\dots , C_{k}^{m_{k}}$ are marked sequents and each top
  rule  $(r_{1}), \dots , (r_{k})$ is
possibly applied more than once.

By the definition of the algorithm, we have hypersequent derivations of
\[ G \hh C_{1}^{1} \quad \dots \quad G \hh
C_{1}^{m_{1}} \qquad \dots \qquad G \hh C_{k}^{1} \quad
\dots \quad G \hh C_{k}^{m_{k}}
\]
for each application of the top rules.
We apply $\textit{Hr}_{\textit{Sys}}$ as follows
\[\infer[] {G \hh \Delta _{1},
\Gamma _{1} \ERGO \Pi _{1} \hh \dots \hh \Delta _{k} , \Gamma _{k } \ERGO \Pi _{k }} 
{M_{1} &&&& \dots  &&&& M_{k }}
\] for each possible combination of $k$ applications of
the top rules $(r_{1}), \dots , (r_{k})$ -- possibly duplicating the
hypersequent derivations previously obtained.
We move the marks to
the conclusions of $(r_1), \dots , (r_k)$.

Notice that we always have hypersequents containing
suitable active components and matching context components. Indeed,
given that we translate into a hypersequent rule application each
possible combination of top rules, at each translation step (above
the bottom rule) we have exactly one hypersequent for
each possible combination of marked sequents.

\item bottom rule:
\[
\infer[(r_{B})]{\Gamma \ERGO \Pi}{\infer*{\Gamma \ERGO \Pi}{} & \infer*{\Gamma \ERGO \Pi}{}}
\]
Without loss of generality we can assume that the top rules of the considered 2-system have been applied
above the premisses of $(r_{B})$ -- as otherwise the application of the 2-system is redundant. 
Hence we have a derivation in $\HJ +\mathbb{H}$
of $G \hh \Gamma \ERGO \Pi \hh \dots \hh \Gamma
\ERGO \Pi$. The desired derivation of $G \hh \Gamma \ERGO \Pi$ is obtained by repeatedly applying $(EC)$.
We move the marks to the conclusion of $(r_{B})$.
\end{enumerate}

\begin{theorem} \label{thm:sys-hyp} For any set $\mathbb{H}$ of
hypersequent rules and set $\mathbb{S}$ of 2-systems s.t.\ if
$\textit{Sys}\in \mathbb{S}$ then $\textit{Hr}_{\textit{Sys}} \in \mathbb{H}$, if $\DER_{\LJ
+\mathbb{S}} \Gamma \ERGO \Pi$ then $\DER_{\HJ +\mathbb{H}} \Gamma
\ERGO \Pi$.
\end{theorem}

\begin{proof} Apply the above algorithm to the $\LJ +\mathbb{S}$
derivation ${\mathcal D}$ of $ \Gamma \ERGO \Pi$ to obtain ${\mathcal D}'$.
The algorithm terminates because the number of rule applications in a
derivation is finite. We show that the
algorithm does not stop before translating the root of ${\mathcal
D}$. The proof is by induction on the number $u$ of 2-system
instances whose top rules are still to be translated.  
If $u=0$ all remaining rules can be translated as soon as the
premisses are marked. Assume $u = n+1$.  
Lemma~\ref{lem:system_division} assures
that there is at least a 2-system
instance $S$ whose top rules are still untranslated and do not occur
below any untranslated top rule. Hence the rule applications that have to be translated
before the top rules of $S$ do not belong to any 2-system and can be
translated as soon as their premisses are marked. After translating
these rules, we can translate the top rules of $S$ and obtain $u = n$.
\end{proof}

\begin{example}
\label{ex:translation}
  The following
  derivation in the calculus
 $\LJ + \textit{Sys}_{(\COM )}$ for G\"{o}del logic (see
Example~\ref{ex:sys_to_hyp})   
\begin{small}
  \[ \infer[(\COM_{B})] {\ERGO ((\varphi \IMPL \psi) \ET (\varphi
      \IMPL \psi) ) \VEL (\psi \IMPL \varphi)}{\infer{\ERGO ((\varphi
        \IMPL \psi) \ET (\varphi \IMPL \psi) ) \VEL (\psi \IMPL
        \varphi)}{ \infer{\ERGO (\varphi \IMPL \psi) \ET (\varphi
          \IMPL \psi)}{ \infer{ \ERGO \varphi \IMPL \psi}{\infer[(\COM
            _{1})'] {\varphi \ERGO \psi}{\psi \ERGO \psi}} &
          \infer{\ERGO \varphi \IMPL \psi }{\infer[(\COM _{1})'']
            {\varphi \ERGO \psi}{\psi \ERGO \psi}}}} & \infer {\ERGO
        ((\varphi \IMPL \psi) \ET (\varphi \IMPL \psi) ) \VEL (\psi
        \IMPL \varphi)}{ \infer{ \ERGO \psi \IMPL
          \varphi}{\infer[(\COM _{2})] { \psi \ERGO \varphi } {
            \varphi \ERGO \varphi}}}}
  \]
\end{small}is translated into the $\HJ +(\COM )$ derivation (see
  Example~\ref{ex:com})
    \[
      \infer[(EC)]{\ERGO ((\varphi \IMPL \psi) \ET (\varphi
\IMPL \psi) ) \VEL (\psi \IMPL \varphi)} {\infer{\ERGO ((\varphi \IMPL \psi) \ET (\varphi
\IMPL \psi) ) \VEL (\psi \IMPL \varphi) \hh \ERGO ((\varphi \IMPL \psi) \ET (\varphi
\IMPL \psi) ) \VEL (\psi \IMPL \varphi)} 
{\infer{\ERGO (\varphi \IMPL \psi) \ET (\varphi
\IMPL \psi) \hh \ERGO  ((\varphi \IMPL \psi) \ET (\varphi
\IMPL \psi) ) \VEL (\psi \IMPL \varphi)} 
{\infer{\ERGO (\varphi \IMPL \psi) \ET (\varphi
\IMPL \psi) \hh \ERGO \psi \IMPL \varphi} { 
\infer{ \ERGO \varphi
\IMPL \psi \hh \ERGO \psi \IMPL \varphi } {\infer{ \varphi
\ERGO \psi \hh \ERGO \psi \IMPL \varphi}{\infer [(\COM )']{\varphi
\ERGO \psi \hh \psi \ERGO \varphi}{\psi
\ERGO \psi & \varphi \ERGO \varphi}}} &  
\infer{ \ERGO \varphi
\IMPL \psi \hh \ERGO \psi \IMPL \varphi } {\infer{ \varphi
\ERGO \psi \hh \ERGO \psi \IMPL \varphi}{\infer [(\COM )'']{\varphi
\ERGO \psi \hh \psi \ERGO \varphi}{\psi
\ERGO \psi & \varphi \ERGO \varphi}}} }}}}
    \]
where $(\COM )'$ translates the pair of top rule applications $\langle
(\COM _{1})' , (\COM _{2}) \rangle $, while $(\COM )''$ translates the pair  $\langle
(\COM _{1})'' , (\COM _{2}) \rangle $.
\end{example}

\subsubsection{Normal forms of $2$-systems derivations}
\label{syst_norm_form}

We introduce the normal forms of 2-system derivations
needed by the algorithm of Section~\ref{tr:sys-hyp} and we show how to
obtain them. The definition of 2-systems (Def.~\ref{def:2systems}) is
indeed decidedly liberal. It allows unrestricted nesting of 2-systems
and does not limit the application of the top rule $(r_{i})$ inside
$\mathcal{D}_{i}$. Such freedom matches naturally the general idea of
a system of rules, but complicates the structure of derivations and the
algorithm for transforming 2-system derivations into hypersequent derivations. We show
below that w.l.o.g. we can consider derivations of a simplified form.

  \begin{lemma}\label{lem:no_same_path}
Any 2-system derivation can be transformed into one with the following
 property: two applications of a top rule $(t)$
belonging to the same 2-system instance never occur on the same
path of the derivation.
  \end{lemma}
  \begin{proof}
Let $\mathcal{P}$ be a 2-system derivation in which 
two applications of $(t)$ occur along the same path, as, e.g., in
\[
 \infer[(t)]{\Delta , \Gamma \ERGO \Pi}{\infer*{\Sigma _{1},
\Gamma \ERGO \Pi}{} & \dots & \infer*[\mathcal{D}]{\Sigma _{i}, \Gamma
\ERGO \Pi}{\infer[(t)]{\Delta , \Gamma ' \ERGO \Pi '}{\infer*{\Sigma _{1} ,
\Gamma ' \ERGO \Pi '}{} & \dots & \infer*{\Sigma _{n} , \Gamma ' \ERGO
\Pi '}{}}} & \dots & \infer*{\Sigma _{n}, \Gamma \ERGO \Pi}{}} 
\]
We use $\WKN$ and $\CTRCT$ to  transform it into
\[ \infer[(t)]{\Delta , \Gamma \ERGO \Pi}{\infer*{\Sigma _{1}, \Gamma
\ERGO \Pi}{} & \dots & \infer=[\CTRCT]{\Sigma _{i}, \Gamma \ERGO
\Pi}{\infer*[\mathcal{D'}]{\Sigma _{i}, \Sigma _{i} , \Gamma \ERGO
\Pi} { \infer=[\WKN]{\Sigma_{i} , \Delta , \Gamma ' \ERGO
\Pi'}{\infer*{\Sigma _{i} , \Gamma ' \ERGO \Pi ' }{}} }} & \dots & \infer*{\Sigma _{n}, \Gamma \ERGO
\Pi}{} }
\]
where for each sequent $\Gamma '' \ERGO \Pi ''$ in $\mathcal{D}$
there is a sequent $\Sigma_{i} , \Gamma '' \ERGO \Pi ''$ in
$\mathcal{D}'$. 
  \end{proof}

Derivations using 2-systems can be further simplified. Indeed the
lemma below shows that we can restrict our attention to derivations
with a limited nesting of $2$-systems.
We use the notion of entanglement to formalise a
violation of this limitation.
 \begin{definition}\label{def:ent}
  Two 2-system instances $S_1$ and $S_2$ are
  {\em entangled} if some top rules of $S_1$ occur above some top
  rules of $S_2$ and some of the former occur below some of the
  latter.
\end{definition}
Consider, for instance, the following derivation schema containing two
2-system instances $a$ and $b$ with bottom rules BOT($a$) and
BOT($b$) and top rules $a_1, a_2$ and $b_1,b_2$, respectively:
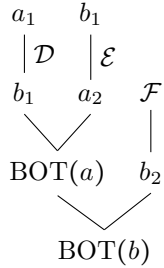
\[
\begin{tikzpicture} [grow'=up]
\Tree [.BOT$(b)$ [.BOT$(a)$ [.$b_{1}$ \edge node[auto=right]{$\mathcal{D}$};
[.$a_{1}$  ]][.$a_{2}$ \edge node[auto=right]{$\mathcal{E}$};
[.$b_{1}$ ]]][.$b_{2}$ [.$\mathcal{F}$ ]]]
\end{tikzpicture}
\]
We use $\mathcal{D}$,
$\mathcal{E}$ and $\mathcal{F}$ to denote derivations.  The
entanglement here occurs because $b_1$ is applied once below $a_1$
and once above $a_2$.
\begin{remark}
  If two 2-system instances are entangled, then all rules of one of
  them occur necessarily above exactly one premiss of the bottom rule
  of the other.
\end{remark}
\begin{example}\label{ex:e-red} 
To disentangle $a$ and $b$, we  make
two copies $b'$ and $b''$ of $b$ that are going to contain the rules formerly belonging to $b$:
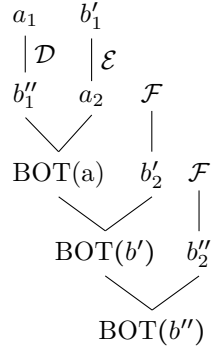
\[
\begin{tikzpicture} [grow'=up]
\Tree [.BOT$(b'')$ [.BOT$(b')$ [.BOT(a) [.$b''_{1}$ \edge
node[auto=right]{$\mathcal{D}$}; [.$a_{1}$  ]][.$a_{2}$ \edge
node[auto=right]{$\mathcal{E}$}; [.$b'_{1}$ ]]][.$b'_{2}$
 [.$\mathcal{F}$ ]]
][.$b''_{2}$  [.$\mathcal{F}$ ]]]
\end{tikzpicture}
\]  The 2-system instances are now disentangled: no top rule of $b'$
occurs below any top rule of $a$ and no top rule of $b''$ occurs above
any top rule of $a$.
\end{example}
The above transformation is the basic step employed in
the following lemma.

\begin{lemma}\label{lem:system_division}
Any 2-system derivation $\mathcal{P}$ can be transformed into a 2-system
derivation $\mathcal{P}'$ of the same end-sequent in which no
entanglement occurs.
\end{lemma}
\begin{proof} 
First we introduce a transformation of derivations
(\emph{e-reduction}) that reduces the number of top rule applications
involved in entanglements. Then we provide a strategy to obtain the
desired derivation $\mathcal{P}' $ using such transformation, and we
prove termination.

\noindent  \emph{E-reduction}: given a 2-system instance $S$ (with bottom rule
$(B_S)$) entangled with 2-system instances $S_{1}, \dots , S_{n}$:
\[\infer[(B_{S})]{\Gamma \ERGO \Pi }{ \infer*{\Gamma
\ERGO \Pi}{ \mathcal{D}_{1}}  & \dots & \infer*{\Gamma
\ERGO \Pi}{ \mathcal{D}_{n}}}
\]
we make two copies $S'$ and $S''$ of $S$ with bottom rules $(B_{S'})$
respectively $(B_{S''})$:
  \[\infer[(B_{S ''})] {\Gamma \ERGO \Pi } { \infer[(B_{S'})] {\Gamma
\ERGO \Pi } {\infer*{\Gamma \ERGO \Pi }{\mathcal{D}_{1}'} &
\infer*{\Gamma \ERGO \Pi }{\mathcal{D}_{2} } & \dots & \infer*{\Gamma
\ERGO \Pi }{\mathcal{D}_{m}}} & \infer*{\Gamma \ERGO \Pi
}{\mathcal{D}_{2}} & \dots & \infer*{\Gamma \ERGO \Pi
}{\mathcal{D}_{n}}}\] in such a way that:
\begin{itemize}
\item if a top rule in $\mathcal{D}_{1}$ belonging to $S$ occurs above
a top rule of one among $S_{1}, \dots , S_{n}$, then its copy in
$\mathcal{D}_{1}'$ belongs to $S'$,
\item if a top rule in $\mathcal{D}_{1}$ belonging to $S$ occurs below
a top rule of one among $S_{1}, \dots , S_{n}$, then its copy  in
$\mathcal{D}_{1}'$ belongs to $S''$.
\end{itemize}

Notice that in the obtained derivation no top rule of $S'$ occurs
below any top rule of $S_{1}, \dots , S_{n}$, and no top rule of $S''$
occurs above any top rule of $S_{1}, \dots , S_{n}$. Moreover, also due to
Lemma~\ref{lem:no_same_path}:

\smallskip

\noindent $(\ast)$ neither $S'$ and $S''$ nor two copies of the same 2-system
instance in $\mathcal{D}_{2}, \dots ,\mathcal{D}_{n}$ can be entangled
or have top rules along the same path of the derivation.

\smallskip

A strategy to apply e-reductions that leads to the required
derivation $\mathcal{P}'$ is the following.
We start reducing one of
the 2-system instances with lowermost bottom rule. Whenever we apply
an e-reduction we collect all entangled copies of the same 2-system
instance in the same class. We continue the disentanglement focusing
on a single class and reducing all its elements before we move on to
another class. Notice that the number of classes never increases and
is bounded by the number of 2-system instances in the original
derivation. Fixed a class, the strategy guarantees that its elements
are disentangled one by one without duplicating other maximally
entangled elements of the same class.

To formalise this strategy let us introduce some auxiliary notions.
We define the equivalence relation $\sim$ as the transitive and symmetric
closure of the binary relation that holds between a 2-system instance
and any of its copies generated by an e-reduction -- notice that
e-reductions do not only copy $S$ but also the 2-system instances in
$\mathcal{D}_{2}, \dots ,\mathcal{D}_{n}$. Given any 2-system
derivation $\mathcal{P}$, let us denote by $E^{\mathcal{P}}$ the set
of all entangled 2-system instances in $\mathcal{P}$, and by
$E^{\mathcal{P}} /_{\sim}$ the quotient set of $E^{\mathcal{P}}$
w.r.t.\ the equivalence relation $\sim$. Moreover, we denote by
$S^{\textit{low}}$ the 2-system instance in $E^{\mathcal{P}}$ which
has the lowest and leftmost bottom rule in $\mathcal{P}$.  Finally, we
compute the \emph{entanglement number} (\emph{e-number} for short) of
a 2-system instance $S$ as follows: for each derivation $\mathcal{D}$
of a premiss of the bottom rule of $S$ we count the number of
equivalence classes containing 2-system instances that have top rules
in $\mathcal{D}$ and are entangled with $S$, then we sum all the resulting
numbers up to obtain the e-number of $S$.

We prove now the statement of the lemma by induction on the
lexicographically ordered triple $\langle \kappa , \mu , \nu
\rangle$  where, fixed the derivation $\mathcal{P}$, 
\begin{itemize} 
\item $\kappa$ is the cardinality of $E^{\mathcal{P}} /
_{\sim}$, i.e.\ the number of classes of entangled
2-system instances, 
\item $\mu$ is the maximum e-number of the elements
of $[S^{\textit{low}}]_{\sim} \in E^{\mathcal{P}} /_{\sim} $,
\item $\nu$ is the number of elements of $[S^{\textit{low}}]_{\sim} \in E^{\mathcal{P}} /_{\sim}$
with e-number $\mu$.
\end{itemize}

\noindent \textbf{\textit{Base case.}}  If either $\kappa$, $\mu$ or $\nu$ are
equal to 0, then no 2-system instance is entangled. Otherwise, first, $E^{\mathcal{P}} /
_{\sim}$ would contain at least one element, and $e \geq 1$. Second,
$[S^{\textit{low}}]_{\sim} \in E^{\mathcal{P}} /_{\sim} $ would not be empty and both $\mu$ and $\nu$
would be greater than 0.

\noindent \textbf{\textit{Inductive step.}}  Given any 2-system
derivation $\mathcal{P}$ with complexity $ \langle \kappa , \mu , \nu \rangle \geq
\langle 1,1,1 \rangle$ we transform it into a 2-system derivation
$\mathcal{P}'$ with complexity smaller than $\langle \kappa , \mu , \nu
\rangle$.  We obtain $\mathcal{P}'$ applying an arbitrary e-reduction
to an uppermost element $S \in [S^{\textit{low}}]_{\sim} \in
E^{\mathcal{P}} / _{\sim} $ with e-number $\mu$.

First notice that we never increase $\kappa$. Moreover, if $\nu > 1$ we reduce
$\nu$ without increasing $\mu$ and if $\nu = 1$ and $\mu > 1$ we reduce
$\mu$. Indeed, after the e-reduction all top rules of $S$ that were
involved in an entanglement with the elements of some class
$[S']_{\sim } \in E^{\mathcal{P}} / _{\sim} $ above the same premiss
of $(B_S)$, are no more involved in such entanglement. This holds
because, due to $(\ast)$ and the definition of $\sim$, the top rules
of elements contained in $[S']_{\sim } \in E^{\mathcal{P}} / _{\sim} $
cannot occur along the same path of the derivation. In general we never increase neither $\mu$ nor
$\nu$, because if we duplicate a 2-system instance during an
e-reduction, either it did not belong to $[S^{\textit{low}}]_{\sim}
\in E^{\mathcal{P}} / _{\sim}$ and hence the copies do not belong to
$[S^{\textit{low}}]_{\sim} \in E^{\mathcal{P}'} / _{\sim}$, or it did
not have maximal entanglement number w.r.t.\ the class
$[S^{\textit{low}}]_{\sim} \in E^{\mathcal{P}} / _{\sim} $, because we
always e-reduce a topmost 2-system instance among those with maximal
e-number in $[S^{\textit{low}}]_{\sim} $. 
Finally, we change the
considered class $[S^{\textit{low}}]_{\sim} $ only when it is empty,
because our e-reduction strategy chooses $S^{\textit{low}}$ only if
$[S^{\textit{low}}]_{\sim} $ is a singleton. If $\nu = 1$ and $\mu = 1$ we
reduce $\kappa$. Indeed, we replace the unique element of
$[S^{\textit{low}}]_{\sim}$ with non-entangled 2-system instances and
$[S^{\textit{low}}]_{\sim}$ does not belong to $E^{\mathcal{P}'} /
_{\sim}$.
\end{proof}

\subsection{From hypersequent to 2-system derivations} 
\label{sec:hyp_to_sys}

Given any set
$\mathbb{H}$ of hypersequent rules and set $\mathbb{S}$ of 2-systems
s.t.\ if $\textit{Hr} \in \mathbb{H}$ then $\textit{Sys}_{\textit{Hr}} \in \mathbb{S}$. Starting
from a derivation in $\HJ +\mathbb{H}$ we
construct a derivation in $\LJ +\mathbb{S}$ of the same
end-sequent.

\subsubsection*{The algorithm}
\label{sec:alg_hyp-sys}

Input: a derivation ${\mathcal D}$ of a sequent $\Gamma \ERGO \Pi$ 
in $\HJ +\mathbb{H}$. Output:
a derivation ${\mathcal D}'$ of $\Gamma \ERGO \Pi$ in $\LJ +\mathbb{S}$.

Intuitively, each application of a $\HJ$ rule in ${\mathcal D}$ is
rewritten as an application of an $\LJ$ rule in ${\mathcal D}'$. Some care
is needed to handle the external structural rules in $\mathbb{H}$ as
well as $(EW)$ and $(EC)$. To deal with the latter rules, which have
no direct translation in $\LJ +\mathbb{S}$, we consider only derivations
${\mathcal D}$ in which $(i)$ all applications of $(EC)$ occur immediately
above the root, and $(ii)$ all applications of $(EW)$ occur where
immediately needed, that is where they introduce components of the context of
rules with more than one premiss. As shown in Section~\ref{sec:prepro}
each hypersequent derivation (of a sequent) can be transformed into an equivalent one
of this form.

The rules in $\mathbb{H}$ are translated in two steps. First for each
component of the premiss of the uppermost application of $(EC)$ in
${\mathcal D}$ we find a \emph{partial derivation}, that is a derivation
in $\LJ$ extended by the top rules of the 2-systems in $\mathbb{S}$
without any applicability condition (Lemma~\ref{lem:tr_H-S}). The
desired derivation ${\mathcal D}'$ is then obtained by suitably applying
to these partial derivations the corresponding bottom rules
(Theorem~\ref{thm:hyp-sys}).

\begin{definition} \label{def:partial} A partial derivation in $\LJ +
\mathbb{S}$ is a derivation in $\LJ $ extended with the top rules
of $\mathbb{S}$ (without their applicability conditions relative to a bottom rule application).
\end{definition}
We show an example of the first part of the translation to
guide the reader's intuition through the proofs that follow.
\begin{example}\label{ex:half_hyp-sys}
Consider the  $\HJ + (\COM )$ derivation
\[\infer[(EC)]{\ERGO (\varphi \ET \psi \IMPL \theta) \VEL (\theta \IMPL \psi
\ET \varphi)}{\infer{\ERGO (\varphi \ET \psi \IMPL \theta) \VEL (\theta
\IMPL \psi \ET \varphi) \hh \ERGO (\varphi \ET \psi \IMPL \theta) \VEL (\theta
\IMPL \psi \ET \varphi)  }{\infer{\ERGO \varphi \ET \psi \IMPL
\theta \hh \ERGO (\varphi \ET \psi \IMPL \theta) \VEL (\theta
\IMPL \psi \ET \varphi)  }{ \infer{\ERGO \varphi \ET \psi \IMPL
\theta \hh \ERGO \theta
\IMPL \psi \ET \varphi}{\infer{ \varphi \ET \psi \ERGO
\theta \hh \ERGO \theta
\IMPL \psi \ET \varphi}{\infer{\varphi \ET \psi \ERGO
\theta \hh \theta \ERGO \psi \ET \varphi}{\infer[(\ET r)]{\varphi , \psi \ERGO
\theta \hh \theta \ERGO \psi \ET \varphi}{\infer[(\COM)']{\varphi , \psi \ERGO
\theta \hh \theta \ERGO \psi }{\infer{\varphi , \theta \ERGO \theta
}{\theta \ERGO \theta} & \psi \ERGO \psi} & \infer[(\COM)'']{\varphi , \psi \ERGO
\theta \hh \theta \ERGO \varphi }{\infer{\theta , \psi \ERGO
\theta}{\theta \ERGO
\theta} & \varphi \ERGO \varphi}}}}} }} }\]
and observe that it satisfies property $(i)$ and,
trivially, property $(ii)$. 
The partial derivations in $\LJ + \textit{\textit{Sys}}_{(\COM )}$
(see Ex.~\ref{ex:sys_to_hyp}) of the components of the uppermost
application of $(EC)$ in the above proof are:
\[
\infer{\ERGO (\varphi \ET \psi \IMPL \theta) \VEL (\theta
\IMPL \psi \ET \varphi)}{\infer{\ERGO \varphi \ET \psi \IMPL
\theta}{\infer{\varphi \ET \psi \ERGO
\theta}{ \infer[\textit{dummy}]{\varphi , \psi \ERGO
\theta}{ \infer[(\COM _{1})']{\varphi , \psi \ERGO
\theta}{\infer{\varphi , \theta \ERGO
\theta}{\theta \ERGO
\theta}}& \infer[(\COM _{1})'']{\varphi , \psi \ERGO
\theta}{\infer{\theta, \psi \ERGO \theta}{\theta \ERGO \theta}}}}}}
\qquad \infer{\ERGO (\varphi \ET \psi \IMPL \theta) \VEL (\theta
\IMPL \psi \ET \varphi)}{\infer{\ERGO \theta \IMPL \psi \ET
\varphi}{ \infer[(\ET r)]{\theta \ERGO \psi \ET
\varphi}{\infer[(\COM _{2})'']{\theta \ERGO \varphi}{\varphi\ERGO
  \varphi} & \infer[(\COM _{2})']{\theta \ERGO \psi }{\psi \ERGO
\psi}}}} 
\] where $(\COM _{1})'$ and $(\COM _{2})'$ translate $(\COM )'$ while
$(\COM _{1})''$ and $(\COM _{2})''$ translate $(\COM )''$.  Notice
that in order to handle the context component duplication relative to
$(\ET r)$, we apply a \emph{dummy} bottom rule.

The  partial derivations obtained have the same structure as the
hypersequent derivations of the corresponding components (see \emph{ancestor tree} in
Def.~\ref{def:ancestors}).
\end{example}

We use Definitions~\ref{def:queue_of_rules} and~\ref{def:form} to
formalise and achieve properties~$(i)$ and~$(ii)$.
\begin{definition}\label{def:queue_of_rules} For any one-premiss rule
$(r)$ we call a {\em queue of $(r)$} any sequence of consecutive
applications of $(r)$ that is neither immediately preceded nor
immediately followed by applications of $(r)$.
\end{definition}

\begin{definition} \label{def:form} We say that an $\HJ + \mathbb{H}$
derivation is in \emph{structured form} \emph{iff} all $(EC)$
applications occur in a queue immediately above the root, and all
$(EW)$ applications occur in subderivations of the form
\[ \infer[(r)]{G \hh C_{0}}{ \infer[(EW)]{G \hh C_{1}}{
\infer*{}{\infer[(EW)]{}{G_{1} \hh C_{1}}} } & \dots & \infer[(EW)]{G
\hh C_{n}}{ \infer*{}{\infer[(EW)]{}{G_{n} \hh C_{n}}} } }
\] where $(r)$ is any rule with more than one premiss and 
each component of $G$ is contained in at least one of the hypersequents $G_{1} , \dots , G_{n}$.
\end{definition} 

A derivation in structured form can be divided into a part
containing only $(EC)$ applications and a part containing the
applications of any other rule. We introduce a notation for the
hypersequent separating the two parts.
\begin{definition}\label{def:premiss_ec} If ${\mathcal D}$ is a derivation
in structured form, we denote by $\widehat{H}_{{\mathcal D}}$ the premiss
of the uppermost application of $(EC)$ in ${\mathcal D}$.
\end{definition}

\begin{definition} \label{def:ancestors} 
Given a $\HJ + \mathbb{H}$ derivation.
A sequent (hypersequent component) $C'$ is a \emph{parent} of a sequent $C$, 
denoted as $p(C, C')$, if one of the following conditions holds:
\begin{itemize}
\item $C$ is active in the conclusion of an application of some $\textit{Hr}
\in \mathbb{H}$, and $C'$ is the active component of a premiss linked to $C$ (see
Definition~\ref{def:linked});
\item $C$ is active in the conclusion of an application of a rule of
$\HJ$, and $C'$ is the active component of a premiss of such application;
\item $C$ is a context component in the conclusion of any rule
application, and $C'$ is the corresponding context component in a
premiss of such application.
\end{itemize} We say that a sequent $C'$ is an
\emph{ancestor} of a sequent $C$, and we write $a(C,
C')$, if the pair $\langle C, C' \rangle$ is in the transitive closure of the
relation $p(\cdot , \cdot)$. The \emph{ancestor tree} of a
sequent $C$ is the tree whose nodes are all
sequents related to $C$ by $a(\cdot , \cdot)$ and whose
edges are defined by the relation $p(\cdot , \cdot)$ between such
nodes.
\end{definition}

We prove below that from any $\HJ +\mathbb{H}$ derivation ${\mathcal
D}$ of a sequent we can construct a partial derivation for each
component of $\widehat{H}_{\mathcal D}$ having the \emph{same
structure} as the ancestor tree of that component, i.e., consisting of the translation
of the rules in the ancestor tree, with the exception of $(EW)$.
\begin{remark}
\
\begin{itemize}
\item In an $\HJ + \mathbb{H}$ derivation that does not use $(EC)$,
the ancestor tree of each hypersequent is a
sequent derivation.
\item If $C$ is the active component of an application of $(EW)$, 
then there is no $C'$ such that $p(C, C')$.
\end{itemize}
\end{remark}
As usual, the \emph{length} of a derivation is the maximal number
of rule applications occurring on any branch plus $1$.

\begin{lemma} \label{lem:tr_H-S} 
Let $\mathbb{H}$ be a set of hypersequent rules and $\mathbb{S}$
of 2-systems s.t.\ if $\textit{Hr} \in \mathbb{H}$ then $\textit{Sys}_{\textit{Hr}} \in \mathbb{S}$.
Given any $\HJ +\mathbb{H}$ derivation ${\mathcal D}$ in structured form,
for each component $C$ of $\widehat{H}_{{\mathcal D}}$ 
we can construct a partial derivation in $\LJ +\mathbb{S}$  
having the same structure as the ancestor tree of $C$ in ${\mathcal D}$. 
\end{lemma}

\begin{proof} 
Let $H$ be a hypersequent in ${\mathcal D}$ derived without using $(EC)$. 
We construct a partial derivation in $\LJ +\mathbb{S}$ with the required 
property for each of its components.
The proof proceeds by induction on the length $l$ of the derivation of $H$ by 
translating each rule of $\HJ +\mathbb{H}$, with the exception of $(EW)$,
into the corresponding sequent rule in $\LJ +\mathbb{S}$.

\noindent \textbf{\textit{Base case.}} If $l = 1$ (i.e.\ $H$ is
an axiom) the partial derivation in $\LJ +\mathbb{S}$ simply contains $H$.

\noindent \textbf{\textit{Inductive step.}} 
We consider the last rule $(r) \not = (EW)$ applied in the
subderivation ${\mathcal D}'$ of $H$, and we distinguish the two cases:
$(i)$ $(r)$ is a one-premiss rule and $(ii)$ $(r)$ has more premisses; for
the latter case, since ${\mathcal D}'$ is in structured form, we deal also
with possible queues of $(EW)$ above its premisses.

\begin{enumerate}
\item Assume that the derivation ending in a one-premiss rule $(r) \in
\HJ$ is
  \[\infer[(r)] {G \hh C'}{\infer*{G \hh C}{{\mathcal D}}}
  \] 
By induction hypothesis there is a partial derivation of $C$ (and of each
component of $G$) having the same structure as the ancestor tree of $C$.
The partial derivation of $C'$ is simply obtained by applying $(r)$.

The case in which $(r)$ is a one-premiss rule belonging to $\mathbb{H}$ is 
a special case of $(ii)$ for which there is no need to consider queues of $(EW)$.
\item Assume that \label{case:hr} $(r)= (\textit{Hr}) \in \mathbb{H} $ has more than one
  premiss, the remaining cases
-- $(r) \in \HJ$, and $(r) \in \mathbb{H}$ and has only one premiss -- being simpler.
 Assume that the derivation ${\mathcal D}'$, of length $n$, is the following
  \[\infer[(\textit{Hr})] {G \hh \Delta_{1}, \Gamma _{1} \ERGO \Pi _{1} \hh \dots
      \hh \Delta _{k}, \Gamma _{k} \ERGO \Pi _{k}}{ \infer*{G
        \hh C ^{1}_{1}}{{\mathcal D}^{1}_{1}} \quad \dots \quad
      \infer*{G \hh C_{1}^{m_{1}}}{{\mathcal D}_{1}^{m_{1}}} \quad
      \dots \quad \infer*{G \hh C_{k}^{1}}{{\mathcal D}_{k}^{1}}
      \quad \dots \quad \infer*{G \hh C_{k}^{m_{k}}}{{\mathcal
          D}_{k}^{m_{k}}}}
  \]
where the premisses $G \hh C^i_j$ of $(\textit{Hr})$ are possibly inferred by a queue of $(EW)$.
When this is the case, we consider the uppermost hypersequents in the queues.
More precisely, we consider the following derivations (each of which has  
length strictly less than $n$) 
  \[
    \infer*{G^{1}_{1} \hh C ^{1}_{1}}{{\mathcal D}^{1}_{1}} \quad
    \dots \quad \infer*{G_{1}^{m_{1}} \hh C_{1}^{m_{1}}}{{\mathcal
        D}_{1}^{m_{1}}} \quad \dots \quad \infer*{G_{k}^{1} \hh
      C_{k}^{1}}{{\mathcal D}_{k}^{1}} \quad \dots \quad
    \infer*{G_{k}^{m_{k}} \hh C_{k}^{m_{k}}}{{\mathcal
        D}_{k}^{m_{k}}}
  \] 
where, for $1 \leq y \leq k$ and $1
\leq x \leq m_{y}$, the hypersequent $G^{x}_{y}$
is $G$ if there is no $(EW)$
application immediately above $G \hh C_{y}^{x}$; otherwise, $G_{y}^{x} \hh
C_{y}^{x}$ is the premiss of the uppermost $(EW)$ application in the
queue immediately above $G \hh C_{y}^{x}$. 

Since ${\mathcal D}$ (and hence ${\mathcal D}'$) is in structured form, each
component of $G$ must occur in at least one of the hypersequents
$G^{1}_{1} , \dots , G_{1}^{m_{1}} , \dots , G_{k}^{1} , \dots ,
G_{k}^{m_{k}} $.  
We obtain partial
derivations for $ \Delta_{1}, \Gamma _{1} \ERGO \Pi _{1} , \dots ,
 \Delta _{k}, \Gamma _{k} \ERGO \Pi _{k}$ applying the top
rules of the 2-system $\textit{Sys}_{\textit{Hr}}$ as follows
\[\infer[(r_{1})]{\Delta _{1} , \Gamma _{1} \ERGO \Pi _{1}}{C_{1}^{1}
\quad \dots \quad C_{1}^{m_{1}}} \qquad \dots \qquad
\infer[(r_{k})]{\Delta _{k}, \Gamma _{k} \ERGO \Pi _{k}}{C_{k}^{1}
\quad \dots \quad C_{k}^{m_{k}}}\]
Indeed, by induction hypothesis, we have a partial derivation for each $C^{x}_{y}$. 
In case a component $C$ of $G$ occurs in more than one premiss, we
have different partial derivations. Hence we apply a dummy bottom rule
\[\infer[] {C} { C & \dots & C}\] 
and obtain one partial derivation.
\end{enumerate}
The obtained partial derivations clearly satisfy the following
property: with the exception of $(EW)$ and of dummy bottom rules, a
rule application occurs in the ancestor tree of a hypersequent
component in ${\mathcal D}$ \emph{iff} its translation occurs in the
partial derivation of such component.
\end{proof}

The next step of the translation consists in applying a bottom rule
for each group of top rules translating one hypersequent rule
application. If we applied dummy bottom rules inside the partial derivations,
we might be forced to apply a single bottom rule for more than one of
such groups -- thus creating what will be called a \emph{mixed
system}. In Theorem~\ref{thm:hyp-sys} we prove that we can always
restructure the derivation and obtain the desired exact match between
groups of top rules and bottom rules. We first show an
example that clarifies the main ideas exploited
in the following proof.
\begin{example}\label{ex:context_problem}
Consider the partial derivations obtained in
Ex.~\ref{ex:half_hyp-sys}, if we apply a bottom rule to them we obtain
the following derivation:
\[
\infer[(\COM _{B})]{\ERGO (\varphi \ET \psi \IMPL \theta) \VEL (\theta
\IMPL \psi \ET \varphi)}{\infer{\ERGO (\varphi \ET \psi \IMPL \theta) \VEL (\theta
\IMPL \psi \ET \varphi)}{\infer{\ERGO \varphi \ET \psi \IMPL
\theta}{\infer{\varphi \ET \psi \ERGO
\theta}{ \infer[\textit{dummy}]{\varphi , \psi \ERGO
\theta}{ \infer[(\COM _{1})']{\varphi , \psi \ERGO
\theta}{\infer{\varphi , \theta \ERGO
\theta}{\theta \ERGO
\theta}}& \infer[(\COM _{1})'']{\varphi , \psi \ERGO
\theta}{\infer{\theta, \psi \ERGO \theta}{\theta \ERGO \theta}}}}}}
& \infer{\ERGO (\varphi \ET \psi \IMPL \theta) \VEL (\theta
\IMPL \psi \ET \varphi)}{\infer{\ERGO \theta \IMPL \psi \ET
\varphi}{ \infer[(\IMPL r)]{\theta \ERGO \psi \ET
\varphi}{\infer[(\COM _{2})'']{\theta \ERGO \varphi}{\varphi\ERGO
  \varphi} & \infer[(\COM _{2})']{\theta \ERGO \psi }{\psi \ERGO
\psi}}}} }
\] where $(\COM _{B})$ is the bottom rule both for $(\COM _{1})'$ and
$(\COM _{2})'$ and for $(\COM _{1})''$ and $(\COM _{2})''$. We call
this a mixed system.

We can abstract this derivation as
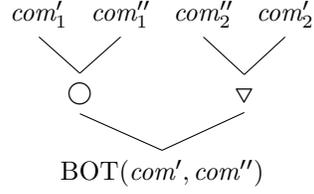
\[
\begin{tikzpicture} [grow'=up]
\Tree [.BOT($\COM ',\COM ''$) 
[.$\bigcirc$ [.$\COM _{1}'$ ] [.$\COM _{1}''$ ]] 
[.$\bigtriangledown$ [.$\COM _{2}''$ ] [.$\COM _{2}'$ ]] ]
\end{tikzpicture}
\]
where we represent by BOT($\COM ',\COM ''$) the bottom rule of $\COM
'$ and $\COM ''$, by $\bigcirc$ the forks in the derivation tree
corresponding to dummy bottom rules, and by $\bigtriangledown$ the forks
corresponding to non-dummy rules.

Given that the removal of premisses from the $\bigcirc$ forks is
a logically sound operation, we transform the structure
of the derivation as follows:
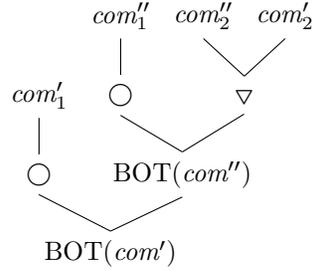
\[
\begin{tikzpicture} [grow'=up]
\Tree [.BOT($\COM '$)  
[.$\bigcirc$ [.$\COM _{1}'$ ]] 
[.BOT($\COM ''$) 
[.$\bigcirc$ [.$\COM _{1}''$ ]] 
[.$\bigtriangledown$ [.$\COM _{2}''$ ] [.$\COM _{2}'$ ]] ]
]
\end{tikzpicture}
\] Now the group of top rules translating $\COM '$ and the one
translating $\COM ''$ have different bottom rules. The
derivation resulting from this is the following
  \[ \infer[(\COM _{B})']{\ERGO \alpha}{ \infer{\ERGO \alpha}{
\infer{\ERGO \varphi \ET \psi \IMPL \theta}{\infer{\varphi \ET \psi
\ERGO \theta} { \infer[(\COM _{1})']{\varphi , \psi \ERGO
\theta}{\infer{\varphi , \theta \ERGO \theta}{\theta \ERGO \theta}} }}
}& \infer[(\COM _{B})'']{\ERGO \alpha}{ \infer{\ERGO \alpha}{
\infer{\ERGO \varphi \ET \psi \IMPL \theta}{\infer{\varphi \ET \psi
\ERGO \theta} { \infer[(\COM _{1})'']{\varphi , \psi \ERGO
\theta}{\infer{\theta, \psi \ERGO \theta}{\theta \ERGO \theta}} }}} &
\infer{\ERGO \alpha}{\infer{\ERGO \theta \IMPL \psi \ET \varphi}{
\infer{\theta \ERGO \psi \ET \varphi}{\infer[(\COM _{2})'']{\theta
\ERGO \varphi}{\varphi\ERGO \varphi} & \infer[(\COM _{2})']{\theta
\ERGO \psi }{\psi \ERGO \psi}}}} }}
  \] where $\alpha$ is the formula $ (\varphi \ET \psi \IMPL \theta)
\VEL (\theta \IMPL \psi \ET \varphi)$.
\end{example}

\begin{theorem} \label{thm:hyp-sys} For any set $\mathbb{H}$ of
hypersequent rules and set $\mathbb{S}$ of 2-systems s.t.\ if $\textit{Hr}
\in \mathbb{H}$ then $\textit{Sys}_{\textit{Hr}} \in \mathbb{S}$, if $\DER_{\HJ
+\mathbb{H}} \Gamma \ERGO \Pi$ then $\DER_{\LJ +\mathbb{S}} \Gamma
\ERGO \Pi$.
\end{theorem}

\begin{proof} 
Let ${\mathcal D}$ be a $\HJ +\mathbb{H}$ derivation of $\Gamma \ERGO
\Pi$. By the results in Section~\ref{sec:prepro} we can assume that
${\mathcal D}$ is in structured form.  By applying the procedure of
Lemma~\ref{lem:tr_H-S} to the premiss $\widehat{H}_{{\mathcal D}}$ of the
uppermost application of $(EC)$ in ${\mathcal D}$ we obtain a set of
partial derivations $\{{\mathcal D}_i\}_{i \in I}$ whose rules translate
those occurring in the ancestor trees of each component of
$\widehat{H}_{{\mathcal D}}$.

We show that we can suitably apply the bottom rules of 2-systems in
$\mathbb{S}$ to the roots of $\{{\mathcal D}_i\}_{i \in I}$ in order to
obtain the required $\LJ +\mathbb{S}$ derivation of $\Gamma \ERGO
\Pi$.  First, we group all top rule applications in $\{{\mathcal D}_i\}_{i
\in I}$ according to the application of $\textit{Hr} \in \mathbb{H}$ that these
rules translate. For each such group we apply one bottom rule below
the partial derivations in which the top rules of the group occur.
As shown in Example~\ref{ex:context_problem}, due to the duplication of context sequents in
hypersequent rules (that we handle using dummy bottom rules), we may need
to apply a single bottom rule below groups of top rules translating
different hypersequent rules. In particular, this happens when a
hypersequent rule application $(r)$ with more than one premiss has an
active component $C_{0}$ and some context components $C_{1},
\dots,C_{n}$ in the conclusion, and two hypersequent rule applications $(h')$ and $(h'')$ have active
components including different ancestors
of some $C_{i}$ with $0 \leq i \leq n$. In this case, the top rules translating
$(h')$ and $(h'')$ occur above different premisses of a non-dummy rule
with conclusion $C_{0}$ (just like the two applications
of $(\COM _{2})$ in Example~\ref{ex:context_problem}) and of some dummy bottom rules with
conclusions $C_{1}, \dots,C_{n}$ (just like the two applications
of $(\COM _{1})$ in Example~\ref{ex:context_problem}). When we apply a bottom rule for such
a group of top rules we obtain a \emph{mixed 2-system}, i.e.\ a
2-system that contains more than one group of top rules translating
different hypersequent rule applications.

We show that we can replace each mixed 2-system by regular
2-systems. First notice that
\begin{enumerate}
\item \label{dot:one_level} two top rule applications belonging to the same mixed
  2-system cannot occur on the same path of the derivation tree,
\item \label{dot:dummy_split} if we remove all premisses but one from a dummy bottom rule in
  a partial derivation we still obtain a partial derivation,
\item \label{dot:dummies_for_non-dummy} every time a pair of top rules
translating different hypersequent rule applications occur in the same
mixed 2-system above different premisses of a non-dummy rule, all other
 pairs of top rules translating these two
hypersequent rule applications occur above different premisses of
dummy bottom rules.
\end{enumerate}
From \eqref{dot:one_level} and \eqref{dot:dummy_split} it follows that if
two top rules occur above different premisses of a dummy bottom rule,
we can remove one of them from the partial derivation containing the
other. If we do so, we say that we \emph{split} the dummy bottom rule.

Consider now a mixed 2-system
\[\infer{\Gamma \ERGO \Delta}{\infer*{\Gamma \ERGO
\Delta}{\mathcal{D}_{1}} & \dots &
\infer*{\Gamma \ERGO \Delta}{\mathcal{D}_{k}}}\] where the derivation
$\mathcal{D}_{i}$, for $1 \leq i \leq k$, contains the
rule applications $(r^{1}_{i}) , \dots , (r^{n}_{i})$. We adopt the
convention that the rules with same superscript index translate the same
hypersequent rule.

To replace such mixed 2-system with regular
2-systems we proceed as follows. First we replace the mixed 2-system with a 2-system for the
group of top rules with superscript 1:
\[\infer[(b^{1})]{\Gamma \ERGO \Delta}{\infer*{\Gamma \ERGO
\Delta}{\mathcal{D}_{1}'} & \dots & \infer*{\Gamma \ERGO
\Delta}{\mathcal{D}_{k}'}}\] where $\mathcal{D}_{1}' , \dots ,
\mathcal{D}_{k}'$ only contain the rules $(r^{1}_{1}) , \dots ,
(r^{1}_{k})$ and those top rules that cannot be removed from the
partial derivations by splitting dummy bottom rules (if we need to
choose, we pick the top rules with minimum superscript index). After
this, we introduce further bottom rules as follows
  \[ \infer[(b^{1})]{\Gamma \ERGO \Delta}{\infer[(b^{2})]{\Gamma \ERGO
\Delta}{\infer*{\Gamma \ERGO \Delta}{\mathcal{D}_{1}'} &
\infer*{\Gamma \ERGO \Delta}{\mathcal{D}_{2}''}& \dots &
\infer*{\Gamma \ERGO \Delta}{\mathcal{D}_{k}''}} & \dots &
\infer[(b^{2})]{\Gamma \ERGO \Delta}{\infer*{\Gamma \ERGO
\Delta}{\mathcal{D}_{1}''} & \dots & \infer*{\Gamma \ERGO
\Delta}{\mathcal{D}_{k-1}''} & \infer*{\Gamma \ERGO
\Delta}{\mathcal{D}_{k}'}}}
\] where the bottom rules $(b^{2})$ are only introduced below the
branches $\mathcal{D}_{1}' , \dots , \mathcal{D}_{k}'$ containing some
of the rules $(r^{2}_{1}) , \dots , (r^{2}_{k})$, and the derivations
$\mathcal{D}_{1}'' , \dots , \mathcal{D}_{k}''$ are copies of
$\mathcal{D}_{1} , \dots , \mathcal{D}_{k}$ only containing
$(r^{2}_{1}) , \dots , (r^{2}_{k})$ and those top rules that cannot be
removed by splitting dummy bottom rules.  We keep duplicating the
derivation in such way until either we do not need any more bottom
rules or we introduced bottom rules for all superscript indices $1,
\dots ,n$. Given that we can add bottom rules for all groups of top
rules in the mixed 2-system, in order to be sure that the result does
not contain any mixed 2-system we only need to show that we never add
a top rule application above the wrong premiss of its bottom rule.
For the sake of contradiction suppose that we do. We add a top rule
application $(r^{i}_{p})$ above a wrong premiss of its bottom rule
only if we just introduced a new bottom rule $(b^{j})$, for $ i < j
\leq n$, and we cannot remove $(r^{i}_{p})$ -- by splitting a dummy
bottom rule -- from the derivation containing a top rule $(r^{j}_{p})$
that we need in the branch that we are considering. But if we cannot
remove $(r_{p}^{i})$ from the partial derivation containing
$(r^{j}_{p})$, by~\eqref{dot:dummies_for_non-dummy} we can remove any
$(r^{j}_{q})$ from any partial derivation containing any
$(r^{i}_{q})$, as long as $q \neq p$. Given that the bottom rule
$(b^{i})$ occurs below $(b^{j})$, it follows that there is no top rule
$(r^{j}_{q})$ on this branch of the bottom rule
$(b^{i})$. By~\eqref{dot:one_level} we can rule out the involvement of
2-system instances different from $i$ and $j$, and hence we can infer
that $(r^{j}_{p})$ is not needed and we do not need to add
$(r^{i}_{p})$ in the first place, contrarily to the assumptions.

Notice that the procedure does not require all groups of top rules to
have exactly $k$ elements. If, for example, the group with superscript
index $i$ contains $l$ top rule applications for $l < k$, then the
bottom rules for $i$ will have $l$ premisses. This does not influence any other group of top rules.

Thus, we eventually obtain an $\LJ +\mathbb{S}$ derivation of $\Gamma
\Rightarrow \Pi$.
\end{proof}

\subsubsection{Normal forms of hypersequent derivations}
\label{sec:prepro}

In the previous algorithm we only considered hypersequent derivations
in structured form, i.e.\ in which $(EC)$ applications occur immediately
above the root and $(EW)$ applications occur where needed.  Here we show
how to transform each hypersequent derivation into a derivation in
structured form.

\begin{definition} The {\em external contraction rank (ec-rank)} of an
application $E$ of $\mathrm{(EC)}$ in a derivation is the
number of applications of rules other than $\mathrm{(EC)}$
between $E$ and the root of the derivation.
\end{definition}

\begin{lemma}
\label{lem:EC} Each  $\mathrm{\HJ }+
\mathbb{H}$  derivation ${\mathcal D}$ can be transformed into a derivation of the same
end-hypersequent in which all $\mathrm{(EC)}$ applications have
ec-rank $0$.
\end{lemma}
\begin{proof} Proceed by double induction on the lexicographically
ordered pair $\langle \mu , \nu\rangle$, where $\mu$ is the maximum
ec-rank of any $\mathrm{(EC)}$ application in ${\mathcal D}$, and
$\nu$ is the number of $\mathrm{(EC)}$ applications in ${\mathcal D}$ with
maximum ec-rank.

\noindent \textbf{\textit{Base case.}} If $\mu = 0$ the claim
trivially holds.

\noindent \textbf{\textit{Inductive step.}} Assume that ${\mathcal D}$ has
maximum ec-rank $\mu$ and that there are $\nu$ applications
of the rule $\mathrm{(EC)}$ with ec-rank $\mu$.  We show how
to transform ${\mathcal D}$ into a derivation ${\mathcal D}'$ having either
maximum ec-rank $\mu ' < \mu$ or ec-rank $\mu $
and number of $(EC)$ applications with maximum ec-rank $\nu '
< \nu$.

Consider an $\mathrm{(EC)}$ application with ec-rank $\mu$
in ${\mathcal D}$ and the queue of $\mathrm{(EC)}$ containing it. There cannot be
any applications of $(EC)$ above this queue because the
ec-rank of its elements is maximal. We distinguish cases
according to the rule $(r)$ applied to the conclusion of the last
element of such queue.

Assume that $(r)$ has one premiss. If
$(r) = (EW)$, we apply $(EW)$ (with the same active
component) before the queue. If $(r) \neq (EW)$, we apply
$(r)$ immediately  before the queue, possibly followed by applications of $(EC)$.

\medskip

\noindent \emph{Notation}. Given a hypersequent $H$ we denote
by $(H)^{u}$ the hypersequent $H \hh \dots \hh H$ containing $u$ copies of $H$ ($u \geq 0$).

\medskip

Let $(r)$ be a(ny external) context-sharing rule with more than one premiss and
consider any subderivation of  ${\mathcal D}$  of the form 
\[ \infer[(r)]{G \hh H }{ \infer[(EC)]{G \hh
C_{1}}{\infer[(EC)]{\vdots}{\infer*{G \hh G'_{1} \hh
(C_{1})^{m_{1}}}{\mathcal{D}_{1}}}} \quad \dots \quad \infer[(EC)]{G
\hh C_{n}}{\infer[(EC)]{\vdots}{\infer*{G \hh G'_{n} \hh
(C_{n})^{m_{n}}}{\mathcal{D}_{n}}}} }
\] 
where $G'_{i}$, for $1 \leq i \leq n$, only contains components in
 $G$ and the derivations $\mathcal{D}_{1},
\dots , \mathcal{D}_{n}$ contain no application of $(EC)$.
We can transform ${\mathcal D}$ into a derivation ${\mathcal D}'$ in which all
applications of $(EC)$ occurring above the hypersequent $G \hh H$ are
either immediately above it or immediately above another application
of $(EC)$; their ec-rank is reduced by $1$ because $(r)$ does not
occur below them anymore.

We first prove that $(\star)$ the hypersequent $G \hh G'' \hh
(H) ^{q}$, where $G'' = G'_{1} \hh \dots \hh G'_{n} $ and $q = (\sum
_{i=1}^{n} (m_{i}-1))+1$ is derivable from \[G \hh
G'_{1} \hh (C_{1})^{m_{1}} \;, \; \dots \; , \; G \hh G'_{n} \hh
(C_{n})^{m_{n}} \] using only $(EW)$ and $(r)$. The hypersequent $G
\hh H$ then follows from $G \hh G'' \hh (H) ^{q}$ by $(EC)$ as all the
components of $G''$ occur also in $G$.  The obtained derivation ${\mathcal
D}'$ has maximum ec-rank $\mu '< \mu $, or the occurrences of $(EC)$
with ec-rank $\mu $ occurring in it are $\nu ' < \nu$.

It remains to prove claim $(\star)$.
We have a derivation of any element of the set
\[ \mathbb{Q} = \lbrace G \hh G'' \hh (H) ^{0} \hh (C_{1})^{x_{1}}
\hh \dots \hh (C_{n})^{x_{n}} \; : \; \sum _{i=1}^{n} x_{i} \; = \;
(\sum _{i=1}^{n} (m_{i}-1))+1 \rbrace
\]
 from the hypersequents 
$G \hh G'_{1} \hh
(C_{1})^{m_{1}} \;, \; \dots \; , \; G \hh G'_{n} \hh
(C_{n})^{m_{n}} $ using only $(EW)$.
Indeed for any hypersequent in $\mathbb{Q}$ and for $1 \leq i \leq
n$, there is at least one $x_{i} \geq m_{i}$, because otherwise $\sum
_{i=1}^{n} x_{i} \; < \; (\sum_{i=1}^{n} (m_{i}-1))+1 $. The claim $(\star)$
therefore follows by Lemma~\ref{lem:downwards} below 
being $G \hh G'' \hh (H) ^{q}$ the only
element of the set
\[ \mathbb{Q}' = \lbrace G \hh G'' \hh (H) ^{q} \hh (C_{1})^{x_{1}}
\hh \dots \hh (C_{n})^{x_{n}} \; : \; \sum _{i=1}^{n} x_{i} \; = \;
0 \rbrace
\]for $q = (\sum _{i=1}^{n} (m_{i}-1))+1$.
\end{proof}

The following is the central lemma of the previous proof.
\begin{lemma}
\label{lem:downwards} For any application of a hypersequent rule
\[ \infer[(r)]{G \hh H }{ G \hh C_{1} && \dots && G \hh C_{n} }
\] and natural number $d \geq 0$, consider the set of hypersequents
\[ \mathbb{L}_{d} = \lbrace G \hh (H) ^{c} \hh (C_{1})^{x_{1}} \hh \dots \hh
(C_{n})^{x_{n}} \; : \; \sum _{i=1}^{n} x_{i} = d \rbrace
\] where $G , H$ are hypersequents, $C_{1} , \dots , C_{n}$ sequents,
and $c$ is a natural number. 
For any natural number $e$, s.t.\ $0 \leq e \leq d$, each element of the set
\[ \mathbb{L}_{(d-e)} = \lbrace G \hh (H) ^{c+e} \hh (C_{1})^{x_{1}'} \hh \dots
\hh (C_{n})^{x_{n}'} \; : \; \sum _{i=1}^{n} x_{i}' =
d-e\rbrace
\] is derivable from hypersequents in $\mathbb{L}_{d}$ by repeatedly applying the
rule $(r)$.
\end{lemma}
\begin{proof} By induction on $e$.

\noindent \textbf{\textit{Base case:}} If $e = 0$, then $\mathbb{L}_{d} =
\mathbb{L}_{d-e}$.

\noindent \textbf{\textit{Inductive step:}} Assume that $e > 0$ and that the claim
holds for all $e'<e$. By induction hypothesis there exists
a derivation from the hypersequents in $\mathbb{L}_{d}$ for each
element of the set
\[ \mathbb{L}_{(d-(e-1))} = \lbrace G \hh (H) ^{c+(e-1)} \hh (C_{1})^{x_{1}''} \hh
\dots \hh (C_{n})^{x_{n}''} \; : \; \sum _{i=1}^{n} x_{i}'' =
d-(e-1)\rbrace
\] that only consists of applications of $(r)$. Any hypersequent
\[G \hh (H) ^{c+e} \hh (C_{1})^{x_{1}'} \hh \dots \hh
(C_{n})^{x_{n}'}\] 
 in $\mathbb{L}_{(d-e)}$ can be derived from elements of $\mathbb{L}_{(d-(e-1))}$ as
follows:
\[ \infer[(r)]{G \hh (H) ^{c+e} \hh (C_{1})^{x_{1}'} \hh \dots \hh
(C_{n})^{x_{n}'}} { G \hh (H) ^{c+(e-1)} \hh H'_{1} \quad \dots \quad G
\hh (H) ^{c+(e-1)} \hh H'_{n} }
\] where, for $1 \leq i \leq n$, 
$H'_{i} \; = \;
(C_{1})^{y_{1}} \hh \dots \hh (C_{n})^{y_{n}}$ is such that if $j\neq i$
then $y_{j} = x'_{j}$ and if $j=i$ then $x'_{j} +1$; i.e.,
the components $ C_{1} , \dots , C_{n} \nin G$ occur in the
$i$\textsuperscript{th} premiss as many times as in the conclusion,
except for $C_{i}$ which occurs one more time.

All premisses of this rule application are hypersequents in
$\mathbb{L}_{(d-(e-1))}$, indeed \[(x_{1}'+1) +x_{2}' +\dots +x_{n}' \;
= \; \dots \; = \; x'_{1} +\dots +x'_{n-1} +(x'_{n}+1) \; = \; (\sum
_{i=1}^{n} x_{i}')+1 \] and \[(\sum _{i=1}^{n} x_{i}')+1 \; = \; (d-e)+1
\; = \; d-(e-1)
\]  Given that only the rule $(r)$ is used to
derive the elements of $\mathbb{L}_{d-(e-1)}$ from the elements of
$\mathbb{L}_{d}$, also the elements of $\mathbb{L}_{(d-e)}$ can be
derived from those of $\mathbb{L}_{d}$ by applying only $(r)$.
\end{proof}

\begin{lemma} \label{lem:push_down_ew} 
Any $\HJ +\mathbb{H}$ derivation of a sequent
can be transformed into a derivation in structured form.
\end{lemma}

\begin{proof} 
Let ${\mathcal D}$ be a hypersequent derivation of a sequent $S$ in $\HJ +\mathbb{H}$. By
Lemma \ref{lem:EC} we can assume that all applications of $(EC)$ 
in ${\mathcal D}$ occur in a queue immediately above $S$.  
Consider an application of $(EW)$, with premiss $G$ and conclusion $G \hh C$, which is not 
as in Definition~\ref{def:form}. First notice that $G \hh C$ cannot be the
root of ${\mathcal D}$. We show how to shift this application
of $(EW)$ below other rule applications until the statement is satisfied for such application.
Three cases can arise:
\begin{enumerate}
\item \label{dot:active_weakening} $C$ is the active component in the
premiss of an application of a rule $(r)$.  The conclusion of $(r)$ is
simply obtained by applying $(EW)$ (possibly multiple times) to $G$.

\item \label{dot:weakening2} $C$ is a context component in the premiss of an application of a
one-premiss rule $(r)$. The $(EW)$ is simply shifted below $(r)$.

\item \label{dot:weakening3} $C$ occurs actively inside the queues of $(EW)$ above all the
premisses of an application of a rule $(r)$.  We remove all
the applications of $(EW)$ with active component $C$ in the queues and
apply $(r)$ with one context component less, followed by $(EW)$.
\end{enumerate} 
The termination of the procedure follows from the fact that ${\mathcal D}$
is finite and that \eqref{dot:active_weakening}--\eqref{dot:weakening3} always reduce the number of rules
different from $(EW)$ occurring below the $(EW)$ applications.
\end{proof}

\section{Applications of the Embeddings}
\label{sec:future}
We provided constructive transformations from hypersequent
derivations to 2-system derivations and back. These transformations
show that the two seemingly different proof frameworks have 
the same expressive power.
The embeddings are not only interesting for
their conceptual outcomes, they also have applications that
are concretely beneficial to both 2-systems and hypersequents.

\subsection{For $2$-systems} 
\label{applications1}
The benefits of the embeddings with respect to $2$-systems include: $(i)$
new cut-free 2-systems, $(ii)$ analyticity proofs, and $(iii)$ locality of
derivations using the hypersequent notation.

$(i)$ and $(ii)$ rely on the method in \cite{Ciabattoni:2008fk} to
transform propositional Hilbert axioms in the language of Full
Lambek calculus into suitable hypersequent rules. 
In a nutshell, the method -- below described for the case of intermediate
logics -- is based on the following classification of 
intuitionistic formulae:
${\mathcal N}_{0}$ and ${\mathcal P}_{0}$ are the set of atomic formulae
\begin{center}
\begin{tabular}{rl}
${\mathcal P}_{n+1}$ & ::= $\FAL \hh \VER \hh {\mathcal N}_{n} \hh {\mathcal P}_{n+1} \ET {\mathcal
    P}_{n+1} \hh  {\mathcal P}_{n+1} \VEL {\mathcal P}_{n+1}$ \\
${\mathcal N}_{n+1}$  & ::= $\FAL \hh \VER \hh {\mathcal P}_{n} \hh {\mathcal N}_{n+1} \ET {\mathcal
    N}_{n+1} \hh  {\mathcal P}_{n+1} \IMPL {\mathcal
    N}_{n+1}$
\end{tabular}
\end{center}
\begin{remark}
The classes $\mathcal{P}_n$ and $\mathcal{N}_n$ contain axioms with leading positive and
negative connective, respectively. Recall that a connective is positive (negative) if its left (right)
logical rule is invertible \cite{Andreoli92}; note that in
$\HJ$, $\vee$ is positive, $\to$ is negative and $\wedge$ is both positive and negative.
\end{remark}

As shown in \cite{Ciabattoni:2008fk} all axioms within the 
class ${\mathcal P}_{3}$ can be algorithmically
transformed into equivalent {\em structural} hypersequent rules that are analytic, i.e.\ 
that preserve cut-elimination when added
to the calculus $\HJ$. For instance 
the rule $(\COM )$ in Example~\ref{ex:com} can be (automatedly\footnote{Program at https://www.logic.at/tinc/webaxiomcalc/})
extracted from the linearity axiom. Furthermore \cite{Ciabattoni:2008fk} shows how to
transform any structural hypersequent rule into an equivalent analytic rule.

Ad $(i)$: the method in
\cite{Negri:2014} rewrites generalised geometric formulae in the
class {\em GA}$_1$ into analytic 2-systems. 
Such formulae follow the schema 
\[\GA _1 \; \equiv \; \forall \overline{x} ( \bigwedge P \IMPL \exists
\overline{y}_1 \bigwedge \GA _0 \VEL \dots \VEL \exists \overline{y}_m \bigwedge \GA _0)\]
Here $\overline{x}, \overline{y}_1, \dots , \overline{y}_m$ are tuples of first order variables, $\bigwedge
P$ is a finite conjunction of atomic formulae, the
variables in $\overline{y}_i$ for any $i$ do not occur free in $\bigwedge P$, and
$\bigwedge \GA _0$ is a finite conjunction of formulae of the form
\mbox{$\forall \overline{x} ( \bigwedge P \IMPL \exists \overline{y}_1
\bigwedge P _1 \VEL \dots \VEL \exists \overline{y}_m \bigwedge P _m)$} 
where the same conditions apply, and $\bigwedge P _j$ is a conjunction
of atomic formulae for any $j$.
As observed in~\cite{Negri:2014}, formulas in $\GA_1$ need not contain quantifier
alternations; indeed there are purely propositional axioms that are in
$\GA_1$ but not in $\GA_0$. Notice that the propositional axioms in
$\GA_1$ are strictly contained in the
class ${\mathcal P}_{3}$ of \cite{Ciabattoni:2008fk}.
For the strictness of the inclusion, consider the axiom
 $\neg \alpha \vee \neg \neg \alpha$.
If we write, as usual, $\neg \varphi$ as $\varphi \IMPL \FAL$, this axiom belongs to $\mathcal{P}_3$ but {\em not}
to {\em GA}$_1$. Hence when applied to $\neg \alpha \vee
\neg \neg \alpha$ the method in~\cite{Negri:2014} does not lead to a
2-system, which can instead be defined
by translating the hypersequent rule equivalent to the axiom (below left)
into the equivalent $2$-system (below right): 
  \[
\vcenter{\infer[(lq)]{G \hh \Sigma \ERGO \hh \Sigma ' \ERGO }{G \hh \Sigma, \Sigma ' \ERGO}} \qquad \qquad \vcenter{\infer{\Gamma \ERGO \Pi}{\infer*{\Gamma \ERGO \Pi}{\infer{\Sigma \ERGO }{}}
&& \infer*{\Gamma \ERGO \Pi}{\infer{\Sigma '  \ERGO}{ \Sigma ,
\Sigma ' \ERGO }}}}
  \]

Ad $(ii)$: The analiticity proof in~\cite{Negri:2014} relies on the fact
that the obtained 2-systems manipulate atomic formulae only; this is
the case for labelled 2-systems arising from frame conditions, but it
does not hold anymore when translating axiom schemata, e.g.\ the axiom
$(\varphi \IMPL \psi ) \VEL ( \psi \IMPL \varphi)$ for G\"odel logic
(cf.\ Example~\ref{ex:com}).  In this case, and for all propositional
Hilbert axioms within the class {\em GA}$_1$, analyticity for the
2-systems obtained by the method in \cite{Negri:2014} can be recovered
by $(a)$ first translating them into hypersequent rules, $(b)$ applying
the \emph{completion} procedure in \cite{Ciabattoni:2008fk} to the
latter, and $(c)$ translating them back. 
\begin{example}
\label{ex:CL}
We show the transformation of a 2-system into an analytic 2-system.
Consider the law of excluded middle $\varphi \vee \neg \varphi \in$ {\em GA}$_1$. 
The method in \cite{Negri:2014} transforms it into
the 2-system (below left), which is translated into the 
  hypersequent rule (below right)
  following the procedure in Section~\ref{sec:proc}: 
  \[
    \vcenter{ \infer{\Gamma \ERGO \Delta}{\infer*{\Gamma \ERGO
          \Delta}{\infer{\Gamma_1 \ERGO \Delta_1}{\varphi , \Gamma_1
            \ERGO \Delta_1}} & \infer*{\Gamma \ERGO
          \Delta}{\infer{\varphi , \Gamma_2 \ERGO \Delta_2}{\FAL ,
            \Gamma_2 \ERGO \Delta_2}}} } \qquad \qquad
    \vcenter{\infer{ G \mid \Gamma_1 \ERGO \Delta_1 \mid \varphi , \Gamma_2
        \ERGO \Delta_2}{G \mid \varphi , \Gamma_1 \ERGO \Delta_1 & G
        \mid \FAL , \Gamma_2 \ERGO \Delta_2}}
  \] Using the results in \cite{Ciabattoni:2008fk} we complete the latter rule and obtain the analytic hypersequent
  rule (below left), whose translation leads to the 2-system
  below right:
  \[
    \vcenter{\infer{G \mid \Gamma_{1} \ERGO \Pi_{1} \mid \Sigma , \Gamma_2 \ERGO \Pi_2
      }{G \mid \Sigma , \Gamma_{1} \ERGO \Pi_{1}}} \qquad \qquad
    \qquad \vcenter{ \infer{\Gamma \ERGO \Pi}{ \infer*{\Gamma \ERGO
          \Pi}{\infer{\Gamma_{1} \ERGO \Pi_{1}}{ \Sigma , \Gamma_{1}
            \ERGO \Pi_{1}}} & \infer*{\Gamma \ERGO \Pi}{\infer{\Sigma
            , \Gamma _2 \ERGO \Pi_2 }{}}} }
  \]
 The analiticity of $\LJ$ extended with the obtained system of rules follows from Theorem~\ref{thm:hyp-sys}.

\end{example}

\subsection{For hypersequent calculi}
\label{naturaldeduction}
We show below how to use the embeddings to reformulate
hypersequent calculi as natural deduction systems inheriting the simplicity
of Gentzen's natural deduction calculus $\NJ$ for intuitionistic logic 
(see, e.g., \cite{Prawitz}). 

Such reformulation is a step forward to prove the connection,
suggested in \cite{Avron:1991}, between intermediate logics formalised
as cut-free hypersequent systems and parallel $\lambda$-calculi.  An
attempt to reveal this connection is the natural deduction calculus
introduced in~\cite{Beckmann&Preining:2015} for G\"odel logic, one of
the main intermediate logics.  Following \cite{HyperAgata}, this
calculus deals with parallel intuitionistic derivations connected by a
symbol $\ast$; this new deduction structure mirroring the hypersequent
separator hinders however the definition of a corresponding
$\lambda$-calculus by Curry--Howard isomorphism.

\medskip

\noindent Our reformulation of hypersequent calculi as natural deduction
systems is modular, and simply obtained by adding to Gentzen's $\NJ$ \emph{higher-level rules} 
simulating hypersequent rules acting on several
components.  The transformation from hypersequent derivations into
2-systems allows us to reformulate the former without using
$\mid$-separated components and without the need of $(EC)$, which is
internalised by the bottom rules of the 2-systems. The resulting
derivations are close to natural deduction.

To present the transformation in a simple way, henceforth we consider 
hypersequent rules of the following form:
\[ \infer[(\textit{Hr})]{ G \mid \Sigma^{1}_{1} , \dots , \Sigma^{1}_{n_{1}} ,
\Gamma_{1} \ERGO \Pi_{1} \mid \dots \mid \Sigma^{k}_{1} , \dots ,
\Sigma^{k}_{n_{k}} , \Gamma_{k} \ERGO \Pi_{k}}{M_{1} &&& \dots &&& M_{k} }
\] where, for any $1 \leq i \leq k $, $M_{i}$ is a (possibly empty) set of hypersequents
of the form $G \mid \Delta^i_j, \Gamma _{i} \ERGO \Pi _{i} $, 
for some $j$, with $\Delta^i_j = \Sigma_{q}^{p}$ for some $1 \leq p \leq k$ and $1 \leq
q \leq n_{p}$, and with $\Gamma_{i} $ and $ \Pi_{i} $ non-empty.

These rules arise by applying the algorithm in \cite{Ciabattoni:2008fk} to
${\mathcal P}_{3}$ formulae (cf.\ the grammar in
Section~\ref{applications1}) of the following form\footnote{In the general case,  ${\mathcal P}_{3}$ formulae correspond to hypersequent rules  with the same form as \emph{Hr} but with more than one $\Delta^i_j$ in each premiss.}:
 \begin{small}
 \[ ((\sigma^{1}_{1} \ET \dots \ET \sigma^{1}_{n_{1}}) \IMPL
 (\delta^{1}_{1} \VEL \dots \VEL \delta^{1}_{m_{1}})) \VEL \dots \VEL
 ((\sigma^{k}_{1} \ET \dots \ET \sigma^{k}_{n_{k}}) \IMPL
 (\delta^{k}_{1} \VEL \dots \VEL \delta^{k}_{m_{k}}))\]
 \end{small}where $\sigma^i_j$ and $\delta^i_j$ are
schematic variables and $(\delta^{i}_{1} \VEL \dots \VEL
\delta^{i}_{m_{i}})$ is $\bot$, if $m_i = 0$. Henceforth we will refer
to this formula as the axiom \emph{associated} to the rule
$(\textit{Hr})$. As shown in \cite{Ciabattoni:2008fk}, $\HJ$ extended
with $(\textit{Hr})$ is {\em equivalent} to $\HJ$ extended with its
associated axiom -- that is, their derivability relations coincide.
 
\begin{example}
 ${\mathcal P}_{3}$ formulae of the above form are, e.g.,
the linearity axiom $(\varphi \IMPL \psi ) \VEL ( \psi \IMPL \varphi)$ 
(see Example \ref{ex:com}), the law of excluded middle, and
the axioms $(Bck)$ characterizing the intermediate logics with $k$ worlds, $k \geq 1$,
$\varphi_0\vee (\varphi_0\to \varphi_1)\vee \dots \vee
(\varphi_0\wedge\dots\wedge \varphi_{k-1} \to \varphi_k)$.
Also the formulae in \cite{L1982} for implicational logics and the disjunctive tautologies in \cite{DanosKrivine} are of this form;
the former paper introduces natural deduction calculi for some
intermediate logics with no normalisation procedure while the latter
interprets the disjunctive tautologies as synchronisation protocols
within the Curry--Howard correspondence framework.
\end{example}

The above hypersequent rule $(\textit{Hr})$ is transformed by the
 embedding in Section~\ref{sec:proc} into the following 2-system
\[
\infer{\Gamma \ERGO \Pi}{\infer*{\Gamma \ERGO \Pi}{\infer[\textit{Tr}_{1}]{ \Sigma^{1}_{1} , \dots , \Sigma^{1}_{n_{1}} ,
\Gamma_{1} \ERGO \Pi_{1} }{M_{1}}} & \dots & \infer*{\Gamma \ERGO \Pi}{\infer[\textit{Tr}_{k}]{\Sigma^{k}_{1} , \dots ,
\Sigma^{k}_{n_{k}} , \Gamma_{k} \ERGO \Pi_{k}}{M_{k}}}}
\]
which is translated into a natural deduction rule $\NR$ of the following form 
\begin{equation}\label{schema:nd} 
\vcenter{
\infer{\varphi}{\infer*{\varphi}{\infer{\varphi_{1}}{\deduce{\sigma^{1}_{1}}{}
      & \dots & \deduce{\sigma^{1}_{n_{1}}}{} & \infer*{\varphi_{1}}{[\delta^{1}_{1}]} &
\dots & \infer*{\varphi_{1}}{[\delta^{1}_{m_{1}}]}}} & \dots &
\infer*{\varphi}{\infer{\varphi_{k}}{\deduce{\sigma^{k}_{1}}{} & \dots &
\deduce{\sigma^{k}_{n_{k}}}{} & \infer*{\varphi_{k}}{[\delta^{k}_{1}]} & \dots &
\infer*{\varphi_{k}}{[\delta^{k}_{m_{k}}]}}}}
}
\end{equation} where $\sigma ^{i}_{j}$ corresponds to $\Sigma
^{i}_{j}$ and $\delta^{i}_{j}$ corresponds to $\FAL$ if $M_{i} =
\emptyset$ and to $\Delta ^{i}_j$ otherwise.

When an upper inference has only one or no $\delta^{i}_{j}$ we we can simplify
the notation as in the following examples.

\begin{remark}
These rules are higher-level rules \`{a} la
Schroeder-Heister~\cite{schroederh2014}, indeed they also discharge rule
applications rather than only formulae. To make this more evident, we denote
them by $*$.
\end{remark}
\begin{example}\label{ex:lin_nd}
The hypersequent rule for the linearity axiom $(\delta \IMPL  \sigma) \VEL (\sigma \IMPL \delta)$ below left (see Example~\ref{ex:sys_com} for the corresponding 2-system) is translated into the natural deduction rule below right:
 \[\vcenter{\infer{G \mid  \delta , \Gamma_{1} \ERGO \Pi_{1} \mid \sigma ,  \Gamma_{2} \ERGO \Pi_{2}  }{G \mid \sigma ,  \Gamma_{1} \ERGO \Pi_{1}  && G \mid \delta ,  \Gamma_{2} \ERGO \Pi_{2} } }  \qquad \qquad \qquad \qquad \vcenter{\infer{\varphi }{ \infer*{\varphi}{\infer{\sigma}{\delta}} && \infer*{\varphi}{\infer{\delta}{\sigma}}}}  \]
Using this rule, the linearity axiom can be derived as follows 
\[ 
\infer[^*]{(\delta \IMPL  \sigma) \VEL (\sigma \IMPL \delta)}{\infer{(\delta \IMPL  \sigma) \VEL (\sigma \IMPL \delta)}{\infer[^1]{\delta \IMPL  \sigma}{\infer[^*]{\sigma}{[\delta]^1}}} & \infer{(\delta \IMPL  \sigma) \VEL (\sigma \IMPL \delta)}{\infer[^2]{\sigma \IMPL \delta}{\infer[^*]{\delta}{[\sigma]^2}}}}
\]

The addition to $\NJ$ of the resulting natural deduction rule yields
the calculus $\NJG$ for G\"odel logic, whose normalisation and
Curry--Howard correspondence have been shown in \cite{lics2017}.
\end{example}
\begin{example}
The hypersequent rule below left for the law of excluded middle
$\sigma \VEL \neg \sigma$ (see Example~\ref{ex:CL} for the
corresponding 2-system) translates into the natural deduction rule
below right:
 \[\vcenter{\infer{G \mid \Gamma_{1} \ERGO \Pi_{1}  \mid \Sigma ,
       \Gamma_2 \ERGO \Pi_2 
}{G \mid \Sigma , \Gamma_{1} \ERGO \Pi_{1} }} \qquad \qquad \qquad
\qquad \vcenter{
\infer{\varphi}{\infer*{\varphi}{\infer{\FAL}{\sigma}} &&
\infer*{\varphi}{[\sigma]}} }
\]
We can derive the law of excluded middle using this rule as follows
\[ 
\infer[^*]{\sigma \VEL \neg \sigma}{\infer{\sigma \VEL \neg \sigma}{\infer[^1]{\neg \sigma}{\infer[^*]{\FAL}{[\sigma]^1}}} & \infer{\sigma \VEL \neg \sigma}{[\sigma]^*}}
\]
\end{example}

We show now that a hypersequent rule $(\textit{Hr})$
and the corresponding natural deduction rule $\NR$ are equivalent, i.e.\ that 
$\DER_{\HJ + \textit{Hr} } \varphi$ if and only if $\DER_{\NJ + \NR} \varphi$.
\begin{theorem}
$\HJ$ extended with any hypersequent rule (\textit{Hr})
is equivalent to $\NJ$ extended with its translated rule $\NR$.
\end{theorem}
\begin{proof}
We show that if $\DER_{\HJ + \textit{Hr} } \varphi$ then
$\DER_{\NJ + \NR} \varphi$.
Indeed a derivation of the
axiom $r_{\alpha}$ associated to $(\textit{Hr})$ is as follows:
  \begin{small}
  \[\infer[^*]{r_{\alpha}}{\infer{r_{\alpha}}{\infer[^{1}]{(\sigma^{1}_{1}
          \ET \dots \ET \sigma^{1}_{n_{1}}) \IMPL (\delta^{1}_{1} \VEL
          \dots \VEL \delta^{1}_{m_{1}}) }{\infer[^2 \, ^*]{\delta^{1}_{1}
            \VEL \dots \VEL
            \delta^{1}_{m_{1}}}{\infer{\sigma^{1}_{1}}{[\sigma^{1}_{1}
              \ET \dots \ET \sigma^{1}_{n_{1}}]^{1}} & \dots &
            \infer{\sigma^{1}_{n_{1}}}{[\sigma^{1}_{1} \ET \dots \ET
              \sigma^{1}_{n_{1}}]^{1}} & \infer{\delta^{1}_{1} \VEL
              \dots \VEL \delta^{1}_{m_{1}}}{[\delta^{1}_{1}]^2 } &
            \dots & \infer{\delta^{1}_{1} \VEL \dots \VEL
              \delta^{1}_{m_{1}}}{[\delta^{1}_{m_{1}}]^2}}}} & \dots &
      \infer*{r_{\alpha}}{}}\]
\end{small}
All hypotheses are derived as shown for the leftmost. The rest of the
premisses of the bottom-most inference are derived similarly using the
implications \[(\sigma^{2}_{1} \ET \dots \ET \sigma^{2}_{n_{2}})
\IMPL (\delta^{2}_{1} \VEL \dots \VEL \delta^{2}_{m_{2}}) \; , \,
\dots \,
, \;
(\sigma^{k}_{1} \ET \dots \ET \sigma^{k}_{n_{k}}) \IMPL
(\delta^{k}_{1} \VEL \dots \VEL \delta^{k}_{m_{k}})\]
The claim
follows by the equivalence between $r_\alpha$ and $(\textit{Hr})$
shown in \cite{Ciabattoni:2008fk}.

To show that if $\DER_{\NJ + \NR} \varphi$ then
$\DER_{\HJ + \textit{Hr} } \varphi$, we derive the rule $\NR$
using the rules of $\NJ$ and $r_{\alpha}$. We can then easily exploit
the equivalence between $\HJ$ and $\NJ$. Intuitively, we use
conjunction and implication elimination to simulate the upper
inferences of $\NR$
(top left part of the following
derivation). Then we nest one disjunction elimination ($\VEL E$) for
each disjunctive subformula of the axiom in order to discharge the
implications used above, discharge the formulae $\delta^{i}_{j}$, and
derive $\varphi, \varphi_1 , \dots , \varphi_k$:
  \begin{footnotesize}
    \[\infer[\VEL E ^{1}]{\varphi}{ \alpha &
        \infer*{\varphi}{\infer[\VEL E
          ^{2}]{\varphi_{1}}{\infer{\delta^{1}_{1} \VEL \dots \VEL
              \delta^{1}_{m_{1}}}{ [(\sigma^{1}_{1} \ET \dots \ET
              \sigma^{1}_{n_{1}}) \IMPL (\delta^{1}_{1} \VEL \dots \VEL
              \delta^{1}_{m_{1}}) ]^{1} & \infer{\sigma^{1}_{1} \ET
                \dots \ET \sigma^{1}_{n_{1}}}{ \sigma^{1}_{1} &
                \infer{\sigma^{1}_{2} \ET \dots \ET
                  \sigma^{1}_{n_{1}}}{\sigma^{1}_{2} &
                  \infer*{\sigma^{1}_{3} \ET \dots \ET
                    \sigma^{1}_{n_{1}}}{}}}} &
            \infer*{\varphi_{1}}{[\delta^{1}_{1}]^{2}} &
            \infer*{\varphi_{1}}{ }}} & \infer*{\varphi}{}}\]
  \end{footnotesize}
The open hypotheses here are the formulae $\sigma^{1}_{1}, \dots ,
\sigma^{1}_{n_{1}} ,$ $\dots ,$ $\sigma^{k}_{1} , \dots ,
\sigma^{k}_{n_{k}}$, which are exactly the hypotheses of
$\NR$.  The claim follows by the equivalence between
$r_\alpha$ and $(\textit{Hr})$ shown in \cite{Ciabattoni:2008fk}.
\end{proof}

\smallskip

\noindent {\bf Final Remark.}  The analiticity of the introduced
natural deduction calculi could be proved by exploiting the connection
with the corresponding cut-free hypersequent calculi.  A computational
interpretation of the former calculi calls however for a direct
normalisation procedure and an interpretation of its reduction rules
as meaningful operations in suitable $\lambda$-calculi.

The case study of G\"odel logic (see Example \ref{ex:com}) has been
detailed in \cite{lics2017}, where we proved normalisation and the
subformula property for its natural deduction calculus $\NJG$ in
Example~\ref{ex:lin_nd}.  Based on this calculus,~\cite{lics2017}
introduces indeed an extension of simply-typed $\lambda$-calculus with
a parallel operator that supports higher-order communications between
processes. The resulting functional language is strictly more
expressive than simply-typed $\lambda$-calculus.

Inspired by hypersequent cut-elimination, the key reductions to prove
the analiticity of $\NJG$ model a symmetric message exchange and
process migration mechanism handling the bindings between code
fragments and their computational environments.

\end{document}